\documentclass[centertags]{amsart}

\usepackage{amssymb}

\usepackage[all]{xy}
\newif\ifPDF
\ifx\pdfoutput\undefined
    \PDFfalse
\else
    \ifnum\pdfoutput > 0
        \PDFtrue
    \else
        \PDFfalse
    \fi
\fi
\ifPDF
    \usepackage[pdftex]{graphicx}
    \DeclareGraphicsExtensions{.pdf,.png,.jpg}
    \graphicspath{{}}
    \usepackage{hyperref}
\else
    \usepackage{graphicx}
    \DeclareGraphicsExtensions{.eps}
    \graphicspath{{}}
   \usepackage[hypertex,colorlinks,raiselinks]{hyperref}
\fi

\newcommand{\R}{\mathbb{R}}
\newcommand{\Z}{\mathbb{Z}}

\newcommand{\mc}{\mathcal}
\newcommand{\circled}[1]{\bigcirc\mspace{-13.0mu}\mathbf{\scriptstyle{#1}}\mspace{5.0mu}}

\newcommand{\eps}{\ensuremath{\epsilon} }

\newcommand{\op}[1]{\operatorname{#1}}

\newcommand{\vol}{\operatorname{Vol}}

\newcommand{\dvol}{\operatorname{dvol}}
\newcommand{\Jac}{\operatorname{Jac}}
\DeclareMathOperator{\tr}{Tr}
\newcommand{\Minvol}{\operatorname{Minvol}}
\newcommand{\ksec}{K_{\operatorname{sec}}}

\newcommand{\CB}{\mathcal B}
\DeclareMathOperator{\Id}{Id}
\DeclareMathOperator{\id}{{\mathbf e}}

\newcommand{\injrad}{\op{injrad} }
\newcommand{\directsum}{\oplus}
\newcommand{\Ric}{\op{Ric} }

\newcommand{\goto}{\xrightarrow}
\newcommand{\D}{\partial }
\newcommand{\pa}{\partial}
\newcommand{\of}{\circ }

\newcommand{\norm}[1]{\left\| #1 \right\| }

\newcommand{\inner}[1]{\left< #1 \right> }
\newcommand{\abs}[1]{\left\lvert #1 \right\rvert }
\newcommand{\set}[1]{\left\{ #1 \right\} }

\newcommand{\tensor}{\otimes}
\newcommand{\grad}{\nabla}
\newcommand{\til}{\widetilde}
\renewcommand{\bar}{\overline}

\newcommand{\la}{\lambda}

\newcommand{\ga}{\gamma}
\newcommand{\Ga}{\Gamma}

\renewcommand{\(}{\left(}
\renewcommand{\)}{\right)}
\renewcommand{\[}{\left[}
\renewcommand{\]}{\right]}

\newcommand{\ds}{\displaystyle}
\newcommand{\scs}{\scriptstyle}

\newtheorem{theorem}{Theorem}[section]
\newtheorem*{theorem*}{Theorem}

\newtheorem{prop}[theorem]{Proposition}

\newtheorem{lemma}[theorem]{Lemma}

\newtheorem{cor}[theorem]{Corollary}

\newtheorem{corollary}[theorem]{Corollary}

\newtheorem{example}[theorem]{Example}

\newtheorem{remark}[theorem]{Remark}
\newtheorem{remarks}[theorem]{Remarks}

\title{Smooth Volume Rigidity for Manifolds with Negatively Curved Targets}

\author[Chris Connell]{Chris Connell$^\dagger$}
\thanks{$\dagger$ Supported in part by an NSF grant DMS-0420432.}% DMS-0608643

\begin{document}

\begin{abstract}
  We establish conditions for a continuous map of
  nonzero degree between a smooth closed manifold and a negatively
  curved manifold of dimension greater than four to be homotopic to
  a smooth cover, and in particular a diffeomorphism when the degree
  is one. The conditions hold when the volumes or entropy-volumes of
  the two manifolds differ by less than a uniform constant after an
  appropriate normalization of the metrics. The results are
  qualitatively sharp in the sense that all dependencies are
  necessary. We present a number of corollaries including a
  corresponding finiteness result. Notably, the method of proof does not rely
  on a $C^\alpha$ or Gromov-Hausdorff precompactness result nor
  on surgery technology.

\end{abstract}

\maketitle

\section{Introduction}
A basic topological question asks when a continuous map of degree
one between two smooth manifolds is homotopic to a diffeomorphism.
In a series of papers (see
\cite{Farrell-Jones:89a},\cite{Farrell-Jones:90},\cite{Farrell-Jones:93}),
Farrell and Jones established their celebrated topological rigidity
result stating that any homotopy equivalence between any closed
manifold and a closed nonpositively curved manifold of dimension at
least $5$ is homotopic to a homeomorphism. However, they also showed
in \cite{Farrell-Jones:89c} that smooth rigidity fails; there are
closed negatively curved Riemannian manifolds $(M,g)$ and $(N,g_o)$
which are homeomorphic but not diffeomorphic. Moreover, for any
$\delta>0$ they have examples where the sectional curvatures of $N$
satisfy $K_{g_o}\equiv -1$ and those of $M$ satisfy $-1-\delta\leq
K_{g} \leq -1$.  In a separate paper, \cite{Farrell-Jones:94c}, they
also gave a set of four criteria, in terms of an ideal boundary
conjugacy, for when a homotopy equivalence between two nonpositively
curved manifold may be realized by a diffeomorphism (see Section
\ref{sec:applications} for details).

The main purposes of this paper is to establish a volumetric condition for the
smooth rigidity of continuous maps with negatively curved targets. We will also
present some generalizations and corollaries.

\begin{theorem}[Volume Gap]\label{thm:vol} Let $f:M\to N$ be any
continuous map between two smooth closed manifolds of dimension
$n>4$. Choose any Riemannian metric $g$ on $M$ normalized to have
sectional curvature bound $ K_{g} \geq -1$, and suppose $N$ admits a
negatively curved metric $g_o$ normalized to have $-\rho^2\leq
 K_{g_o} \leq -1$. There is a constant $C>0$ such that if
$$\vol_{g}(M)\le \abs{\deg(f)}\vol_{g_o}(N)+{C},$$ then $f$ is
homotopic to a smooth covering map of degree $\abs{\deg(f)}$. Moreover, $C$
depends only on  $n$, the injectivity radius of $(M,g)$, the pinching constant
$\rho\geq 1$ and  $\abs{\deg(f)}\norm{N}$, where $\norm{N}$ is the
simplicial volume of $N$.
\end{theorem}

\begin{remarks}\label{rem:theorem1}\noindent
\begin{itemize}
\item The constant $C$ always satisfies $C<\vol_{g}(M)$ so that the volume
constraint is never satisfied when $f$ has degree zero. When the degree is
not zero, the resulting local diffeomorphism is given by an explicit
construction from the original continuous map $f$. We will show with some
examples (see \ref{ex:conditions}) that the dependence of the constant $C$
on both the injectivity radius of $(M,g)$ and $\rho$ is necessary. Also, the
injectivity radius dependency can be exchanged for a lower bound on the
normalized volume of sufficiently small balls such as $\inf_{p\in M}\inf_{0<
r\leq 1}\frac{1}{r^n}\vol_{g}(B(p,r))$.

\item  Note that $\norm{N}$ only depends on $\pi_1(N)$ and in even
dimensions we may replace the dependence of $C$ on
$\norm{M}$ by $\abs{\chi(N)}$.  Also, we will see that the dependence of
$C$ on $\abs{\deg(f)}\norm{N}$ can also be exchanged for a dependence
on $\abs{\nabla Rm(g_o)}$, where $Rm$ is the curvature operator on
$\Lambda^2TN$. Thus one may remove the dependence of $C$ on
$\abs{\deg(f)}$ if needed.

\item Under the hypotheses of Theorem \ref{thm:vol}, Besson, Courtois and
Gallot proved in \cite{Besson-Courtois-Gallot:98} that
$$\vol_{g}(M)\geq \vert\deg(f)\vert\vol_{g_o}(N)$$ with equality
being achieved if and only if $N$ and $M$ both have constant curvature $-1$
and $f$ is homotopic to a Riemannian cover. From this point of view,
Theorem \ref{thm:vol} can be viewed as a coarse (topological) version of
their result.

\item If $(M,g)$ satisfies the hypotheses of Theorem \ref{thm:vol} for some
$(N,g_o)$ and fixed value of the constant $C>0$, then $(M,g)$ has Ricci
curvatures bounded below, and injectivity radius and volume bounded
above. Hence by Theorem 0.2 of \cite{Anderson-Cheeger:92a}, it was already
known that there are at most a finite number of possible diffeomorphism
types for such $(M,g)$.

\item Bessi\`{e}res (\cite{Bessieres:98}) first established the special case of
Theorem \ref{thm:vol} when $C=0$ and $(N,g_o)$ is hyperbolic. Specifically,
he extends the main result of \cite{Besson-Courtois-Gallot:95} to show that if
$f:M \to N$ is a map of nonzero degree with $N$ hyperbolic and
$$\Minvol(M)= \abs{\deg(f)}\Minvol(N)=\abs{\deg(f)}\vol_{g_o}(N),$$ then
$M$ admits a hyperbolic metric and $f$ is homotopic to a smooth cover of
degree $\deg(f)$. Moreover in \cite{Bessieres:00a}, he produces an example
of a noncompact finite volume hyperbolic manifold $N$ and a manifold $M$,
not homeomorphic to $N$, together with a degree one map $f:M\to N$ such
that $\Minvol(M)\leq \Minvol(N)$ and their simplicial volumes coincide,
$\norm{M}=\norm{N}$. Hence, a finite volume version of Theorem
\ref{thm:vol} must address more than just a suitable replacement for the
dependence of $C$ on the injectivity radius.
\end{itemize}
\end{remarks}

We will derive Theorem \ref{thm:vol} as a special case of two other
progressively more general results. For any finite volume Riemannian
manifold $(M,g)$, define the {\em volume growth entropy} of the
metric $g$ to be,
$$h(g)=\limsup_{R\to\infty}\frac{\log \vol_{g}(B(x,R))}{R}$$
where $B(x,R)$ is the ball of radius $R$ in the Riemannian universal
cover $\til{M}$ about $x\in \til{M}$. The definition is independent
of $x$. Moreover Manning showed that for $M$ closed and
nonpositively curved, the limit always exists and equals the
topological entropy of the geodesic flow on $M$ (\cite{Manning:79}).

For a negatively curved Riemannian manifold $(N,g_o)$ we can
consider the quantity
$$u(g_o)=\inf_{x\in N} \inf_{\la\in \mc{P}(S_xN)}\sqrt{n}\(\frac{\det
\la\(\op{Hess}_x(B_v)\)}{\sqrt{\det\la\(v\tensor
v\)}}\)^{\frac{1}{n}}$$
where $\mc{P}(S_xN)$ is the space of probability measures on the
unit tangent sphere $S_xN$, $v\in S_xN$ is the variable of
integration by $\lambda$ and $B_v$ is the Busemann function
associated to $v$. There are four important properties of the
quantity $u(g_o)$ which can be easily derived from the work in
\cite{Besson-Courtois-Gallot:99}: it scales the same way the entropy
does, namely $u(c\cdot g_o)=\sqrt{c}\ u(g_o)$, it satisfies
$u(g_o)\geq a(n-1)$ (resp. $u(g_o)\leq b(n-1)$) whenever $ K_{g_o}
\leq -a^2$ (resp. $ K_{g_o} \geq -b^2$), $h(g_o)\geq u(g_o)$ and if
$g_0$ is a locally symmetric metric, then $u(g_0)=h(g_0)$.

We now establish some notation for what follows. For any Riemannian
manifold $(M,g)$,  we denote its injectivity radius by $\injrad(g)$ and  its
universal cover by $(\til{M},\til{g})$. Set
$\kappa(g)=\sqrt{\abs{\inf_{P\in\op{Gr}_2(TM)} K_{g}(P) }}$. Whenever
$\kappa(g)=0$, we have $\Ric(g)\geq 0$ and so $h(g)=0$. Therefore  by the
aforementioned result of \cite{Besson-Courtois-Gallot:95}, if $N$ admits a
negatively curved metric and there is a map $f:M\to N$ of nonzero degree, then
$\kappa(g)>0$. For any negatively curved Riemannian manifold $(N,g_o)$ we
define the {\em pinching constant} to be
$\ds{\rho(g_o)=\sqrt{\frac{\inf_{P\in\op{Gr}_2(TN)} K_{g_o}(P)
}{\sup_{P\in\op{Gr}_2(TN)} K_{g_o}(P) }}}.$ We now state the normalization
free version of Theorem \ref{thm:vol}.

\begin{theorem}[Smooth Entropy-Volume Rigidity]\label{thm:ent-vol}
Let $f:M\to N$ be a continuous map of nonzero degree between any closed
Riemannian manifold $(M,g)$ and a closed negatively curved manifold
$(N,g_o)$ of dimension $n>4$. There is a constant $C$ depending only on
$\frac{h(g)}{\kappa(g)}$, $\kappa(g)\cdot \injrad(g),$
$\abs{\deg(f)}\norm{N}$ and $\rho(g_o)$ such that if

$$h(g)^n\vol_{g}(M)\le \abs{\deg(f)}u(g_o)^n\vol_{g_o}(N)+C,$$
then $f$ is homotopic to a smooth covering map of degree $\abs{\deg(f)}$.
\end{theorem}

\begin{remark}
Here the quantities $C$, $\frac{h(g)}{\kappa(g)}$, $\kappa(g)\cdot \injrad(g)$,
$\rho(g_o)$, $h(g)^n\vol_{g}(M)$ and $u(g_o)^n\vol_{g_o}(N)$ are all
invariant under scaling either of the metrics $g$ or $g_o$. Moreover, $C$
necessarily tends to $0$ if either $\kappa(g)\cdot\injrad(g)$ or
$\frac{h(g)}{\kappa(g)}$ tends to zero or if
$\rho(g_o)$  tends to infinity. %
(see \ref{ex:conditions}).

\end{remark}

For any closed orientable topological $n$-manifold $N$ admitting a metric of
negative curvature, we let $\mc{M}_{\delta,k}(N)$ be the family of Riemannian
$n$-manifolds $(M,g)$ with $\kappa(g)\injrad(g)>\delta$ and admitting a
degree $k$ continuous map to a fixed topological manifold $N$. Similarly we
define $\mc{N}_{n,\rho}$ to be the family of closed n-manifolds $(N,g_o)$ with
$-a^2\rho^2\leq K_{g_o}\leq -a^2$ for any $a>0$.

We can optimize each side of the inequality in Theorem \ref{thm:ent-vol} over
any smooth equivalence class of metrics as follows. Suppose $M_\phi$ and
$N_{\phi_o}$ represent topological $n$-manifolds $M$ and $N$ for $n>4$
equipped with two specific smooth structures $\phi$ and $\phi_o$ respectively.
By passing to subsequences, we may always choose a sequence $\set{g_i}$ of
metrics achieving the infimum,
$$\inf_{(M_\phi,g)\in\mc{M}_{\delta,k}(N)}h(g)^n\vol_{g}(M)$$
such that the limits $\underline{\vol}_\delta(M_\phi)=\lim_i \vol_{g_i}(M)$ and
$\underline{h}=\lim_i h(g_i)$ both exist. Similarly define the supremum of the
volumes of $N$ by metrics $g_o$ rescaled so that $u(g_o)=\underline{h}$ to be
$$\bar{\vol}_\rho(N_{\phi_o})=\sup_{\set{g_o\,|\,(N_{\phi_o},g_o)\in
\mc{N}_{n,\rho}\ \text{and}\ u(g_o)=\underline{h}}}\vol_{g_o}(N).$$

The following is an immediate corollary of Theorem \ref{thm:ent-vol} and
Remarks \ref{rem:theorem1}.

\begin{corollary}\label{cor:supvol}
For given smooth topological manifolds $M_\phi$ and $N_{\phi_o}$ of
dimension $n>4$ as above, there is a constant $C>0$ depending only on
$\delta$ and $\rho$ such that $M_{\phi}$ is diffeomorphic to a degree $k$
cover of $N_{\phi_o}$ if
$$\underline{\vol}_\delta(M_\phi)\leq k\ \bar{\vol}_\rho(N_{\phi_o})+C.$$ %In

\end{corollary}

The following theorem of Gromov (1.7 of \cite{Gromov:78b}) shows that the
additive curvature pinching constant, $\eps$, in the Farrell and Jones examples
must depend on the volume of $N$.
\begin{theorem}\label{thm:Gromov}
For $(M,g)$ closed of dimension $n\geq 4$, there is a $\eps>0$ depending only
on an upper bound for $\vol_{g}(M)$ such that if $-1-\eps\leq  K_{g} \leq -1$,
then $M$ is diffeomorphic to a hyperbolic manifold.
\end{theorem}

The following corollary of Theorem \ref{thm:vol} is an equivalent statement of
the above theorem in the $n>4$ case, but by an alternate proof which we will
provide in Section \ref{sec:applications}.
\begin{corollary}\label{cor:pinch}
For $(M,g)$ closed of dimension $n>4$ there is a $\eps>0$ depending only on
$\norm{M}$ such that if $M$ has pinched curvatures $-1-\eps\leq K_g\leq -1$,
then $M$ is diffeomorphic to a hyperbolic manifold.
\end{corollary}

This result implies the previous one since since by 1.4 of
\cite{Gromov:78b} there are only a finite number of possible
diffeomorphism types under the assumptions of Theorem
\ref{thm:Gromov}. In fact, \ref{thm:Gromov} implies \ref{cor:pinch}
since, under the assumptions of the Corollary, a theorem of
Thurston's (see Sections 0.3 and 1.2 of \cite{Gromov:82}) implies
$\vol_g(M)$ is bounded by a uniform constant times the simplicial
norm of $M$ which is a homotopy invariant. In particular, the
corollary is known to hold in dimension $4$ as well. These results
are false in dimension $3$, as shown by the examples of homotopy
inequivalent manifolds with bounded volumes and curvatures tending
to $-1$ found in \cite{Gromov:78b}. In fact, these examples can be
chosen to be hyperbolic by the work of Thurston \cite{Thurston:77}.

Both the examples of Gromov and Thurston \cite{Gromov-Thurston:87}
in dimension $n\geq 4$ and the counterexamples of Farrell and Jones
(\cite{Farrell-Jones:94}) and Farrell, Jones and Ontaneda
(\cite{Farrell-Jones-Ontaneda:98b}) in dimension $n>4$ mentioned
earlier show that $\delta$ in the above corollary must depend on
$\pi_1(M)$. For the Gromov-Thurston examples in dimension $n>4$, we
see this dependence explicitly since $\delta<\frac{C}{\log i}$ where
$i$ is the degree of the ramified covers over a fixed manifold which
they use as their examples. The same statement in dimension $n\geq
4$ can be derived directly from Theorem \ref{thm:Gromov} in
conjunction with Wang's finiteness theorem \cite{Wang:72}.

Another principal feature of Theorems \ref{thm:vol} and
\ref{thm:ent-vol} is that they do not rely on a Cheeger-Gromov type
compactness theorem. In fact, even the family of closed Riemannian
manifolds $M$ with fixed $\pi_1(M)$, curvatures and injectivity
radius bounded below, and admitting a map of nonzero degree onto a
fixed negatively curved manifold is not precompact in the
Gromov-Hausdorff topology, since one may metrically connect sum any
such $M$ with a a sufficiently large dilation of an arbitrary simply
connected closed nonnegatively curved manifold and stay within this
family. As such, we will indicate how we can sometimes use Theorem
\ref{thm:vol} to replace Anderson and Cheeger-Gromov type
compactness arguments (e.g.
\cite{Cheeger:69,Cheeger:70a,Anderson-Cheeger:91a,Anderson-Cheeger:92a})
to obtain smooth topological finiteness results. For instance, if we
fix the topology, then we have the following smooth finiteness
theorem of Belegradek (see also Fukaya \cite{Fukaya:84}).

\begin{theorem}[\cite{Belegradek:02}]\label{thm:Belegradek}
For $n\geq 3$ and constants $b\geq a>0$, there are only a finite
number of diffeomorphism types of finite volume manifolds with fixed
$\pi_1$ and $-b^2\leq \ksec \leq -a^2$.
\end{theorem}

In Section \ref{sec:applications} we will show that this theorem, in
the case of closed manifolds of dimension $n>4$, follows from
Theorem \ref{thm:ent-vol}. We will also prove there a generalization
of this to the following finiteness theorem which arises as a
corollary of Theorem \ref{thm:vol}.

\begin{corollary}\label{cor:finiteness}
Fix any topological manifold $N$ of dimension $n>4$ admitting a negatively
curved metric, and let
$$V(\delta)=\sup\set{u(g_o)^n\vol_{g_o}(N)+C(n,\delta,\rho)\,:\,
\rho\geq 1 \text{ and } \ (N,g_o)\in\mc{N}_{n,\rho}},$$ where
$C(n,\delta,\rho)$ is the constant from Theorem \ref{thm:ent-vol} for
$\abs{\deg(f)}=1$. If $\mc{M}$ represents the class of all Riemannian
manifolds $(M,g)$ admitting a degree one map to $N$ and satisfying the
entropy-volume bound $h(g)^n\vol_{g}(M)<V(\delta)$, then $\mc{M}$ has
only a finite number of diffeomorphism types.
\end{corollary}

Now we state our most general theorem. For a $C^1$ map $F:M\to N$ between
two manifolds let $T_r(F)\subset M$ be the $r$-tubular neighborhood of the
critical points of $F$, i.e.
$$T_r(F)=\bigcup_{\set{x\,|\,\op{Jac}_x(F)=0}}B(x,r).$$
Theorems \ref{thm:vol} and \ref{thm:ent-vol} are special cases of
the following more general theorem.
\begin{theorem}\label{thm:critical}
Let $f:M\to N$ be a continuous map of nonzero degree between any closed
Riemannian manifold $(M,g)$ and a closed negatively curved manifold
$(N,g_o)$ of dimension $n>4$. There exists a $C^1$ map $F:M\to N$ homotopic
to $f$ and a number $r$ depending only on $n$, $\rho,$
$\abs{\deg(f)}\norm{N}$ and $\kappa(g)\cdot\injrad(\til{g})$ such that
$$h(g)^n\vol_{g}(M)\geq u(g_o)^n \abs{\deg(f)}\vol_{g_o}(N) + \frac12 h(g)^n
\vol_{g}\(T_{\frac{r}{\kappa(g)}}(F)\).$$
\end{theorem}

Theorem \ref{thm:critical} implies  that adding ``smooth topology'' to $M$
uniformly increases its volume. For instance, starting with $M=N$ and adding
$k$ $i$-handles, for any $i=1,\dots,n-1$, with bounded normalized injectivity
radius to $M$ increases the entropy-volume of $M$ by at least $kC$ for a fixed
constant $C$. This follows from the injectivity radius bound, since for the
resulting degree $1$ map $F$ there must be at least $k$ critical points, one on
each handle, separated by a distance of at least the injectivity radius. This
bound could be made sharper by taking better account of the entire critical
locus for handles. It is generally easy to detect the topological change resulting
from adding handles. However, there are many more subtle ways of changing
smooth topology. A less intuitive example would be to keep the topology fixed,
and allow changes to the smooth structure. If $(M_1,g_1)$ and $(M_2,g_2)$ are
two homeomorphic, but nondiffeomorphic, negatively curved manifolds, then
we cannot have $u(g_i)=h(g_i)$ for both $i=1,2$. Otherwise either the above
inequality holds, or we could reverse the roles of $M_1$ and $M_2$ so that it
holds. There are some other general situations where we automatically have a
degree one map. The following corollary gives one such example.

\begin{corollary}\label{cor:connect}
  Let $(N,g_o)$ be closed with $-\rho^2\leq K_{g_o}\leq -1$. For any smooth manifold $Q$
  and metric $g$ on $N\# Q$ rescaled so that
 $K_g\geq -1$, there is a constant $C=C(n,\injrad(g),\rho)$ such that if
  $$\vol_{g}(N\#Q)\leq \vol_{g_o}(N)+C,$$ then $N\#Q$ is
  diffeomorphic to $N$.
\end{corollary}

In Section \ref{sec:prelim} we recall the construction of the
generalized natural maps $F_s$ due to Besson, Courtois and Gallot. %In Section 3 we use the
There we also reduce the proof of Theorems
\ref{thm:vol},\ref{thm:ent-vol} and \ref{thm:critical} to a key
estimate. In Section \ref{sec:Jac} we derive the main components of
our main estimate, and in Section \ref{sec:synthesis} we put these
together. Finally, in Section \ref{sec:applications} we prove the
remaining corollaries and some additional related results.
\medskip

\section*{Acknowledgements}

The author would like to thank Benson Farb for helpful comments.

\section{Preliminaries}\label{sec:prelim}

Let $(M,g)$ and $(N,g_o)$ be closed, orientable manifolds and let
$f:M\to N$ be a degree $d$ map. Since the quantities in the
inequality of Theorem \ref{thm:critical} are scale invariant in both
$g$ and $g_o$, we will from now on, unless otherwise stated, assume
that we have scaled the metrics so that the sectional curvatures of
${N,g_o}$ are bounded from above by $-1$ and that those of $(M,g)$
from below by $-1$. This normailization removes two extra parameters
that we would otherwise have to drag around.

We begin by recalling the construction of the natural maps
$$F_s:M\to {N}$$
due to Besson, Courtois and Gallot in its present form. Let $\til f:
\til M\to\til {N}$ denote the lift of $f$ to the universal covers.
For each $s>0$ and $x\in \til M$ consider the measure $\mu_x^s$ on
$\til M$ in the Lebesgue class with density
$$\frac{d\mu_x^s}{\dvol_g}(z)=e^{-sd(x,z)}$$
where $d$ is the distance function of $\til M$. Recall the
definition of the volume growth entropy $h(g)$. For all $s>h(g)$ and
all $x\in\til M$ the total measure $\Vert\mu_x^s\Vert$ of $\mu_x^s$
is finite.

Consider the push-forward measure $\til f_*\mu_x^s$ on $\til {N}$,
and define a measure $\sigma_x^s$ on $\D\til {N}$ in the following
way. For $z \in \til{{N}}$, let $\nu_z$ be the ``visual'' or
Patterson-Sullivan measures normalized to be probability measures on
$\D\til {N}$ (see \cite{Besson-Courtois-Gallot:95}), and for
$U\subset\D\til {N}$ measurable define
\begin{equation}\label{convolution}
\sigma_x^s(U)=\int_{\til {N}}\nu_z(U)d(\til f_*\mu_x^s)(z).
\end{equation}
That is, we take $\sigma_x^s$ to be convolution of the push-forward
measure $\til f_*\mu_x^s$ with the visual measures $\nu_z$. Notice
that for all $s,x$, $\Vert\mu_x^s\Vert=\Vert\sigma_x^s\Vert$, so the
measure $\sigma_x^s$ is finite for $s>h(g)$.

For $\theta\in\D\til {N}$ denote by $B_\theta(y)$ the Busemann
function of ${N}$ (normalized so that $B_\theta(O) = 0$ for some
fixed origin $O \in \til{{N}}$) and consider the function on $\til
{N}$ defined by
\begin{equation}\label{function}
\mathcal{B}_{\sigma_x^s}(y)=\int_{\D\til
{N}}B_\theta(y)d\sigma_x^s(\theta).
\end{equation}
This is a proper strictly convex function, hence it has a unique
minimum \cite{Besson-Courtois-Gallot:95}, which we call the
\textit{barycenter} of the measure $\sigma_x^s$ and denote by
$\op{Bar}(\sigma_x^s)$.

This construction is much more general: Given any finite measure
$\lambda$ on $\til {N}$ we can define as in \eqref{convolution} a
measure $\sigma_\lambda$ as the convolution of $\lambda$ with the
family of visual measures. For example the convolution of the
Dirac-measure $\delta_z$ with support $z\in\til {N}$ is the visual
measure $\nu_z$. Similarly, we can define for every finite measure
$m$ of $\D\til{{N}}$ the function $\CB_m$ as in \eqref{function}.
The function $\CB_m$ is proper and convex if $m$ has no atoms. If
this is the case, we define $\op{Bar}(m)$, the barycenter of $m$, to
be the unique minimum of $\CB_m$.

For all $s>h(g)$, the map $\til{F}_s:\til{M}\to\til{{N}}$ defined by
$x\mapsto \op{Bar}(\sigma_x^s)$ is equivariant under the action of
$\pi_1(M)$ and $\pi_1({N})$ and so descends to the \emph{natural
map} $F_s:M \to {N}.$ The following is a collection and restatement
of some of the important properties of the natural map due to
Besson, Courtois and Gallot
\cite{Besson-Courtois-Gallot:95,Besson-Courtois-Gallot:98}. In the
statements found there, the authors used $h(g_o)$ instead of
$u(g_o)$ in the case the target is a locally rank one symmetric
space or else $n-1$ for the case when the target is negatively
curved with upper curvature bound $-1$. However, their method of
proof was to show the following more general version, and then show
separately that $h(g_o)= u(g_o)$ when $g_o$ is locally symmetric and
that $n-1\geq u(g_o)$ when $ K_{g_o} \leq -1$.

\begin{theorem}\label{bcg}
Let $(M,g)$ and $({N},g_o)$ be closed orientable manifolds, let
$f:M\to {N}$ be a map of nonzero degree and assume that the
sectional curvature of ${N}$ is pinched and bounded from above by
$-1$. For all $s>h(g)$ and all $x\in M$,
\begin{enumerate}
\item The natural maps $F_s$ are at least $C^1$.
\item The map
    $\til\Psi_s: [0,1]\times\til M\to\til {N}$ defined by
        $\til\Psi_s(t,x)=\op{Bar}\(t\nu_{\til f(x)}+
        (1-t)\sigma_x^s\)$
    is equivariant and induces a continuous homotopy between $f$
    and $F_s$.
\item   $\vert \op{Jac}(F_s)(x)\vert\le \left(\frac s{u(g_o)}\right)^n$.
\end{enumerate}
\end{theorem}

\begin{remark}
  The appropriate version of the above theorem also holds when $M$
  or $N$ are not orientable, assuming that $f$ induces an
  orientation true homomorphism between the fundamental groups.
\end{remark}

The above theorem shows that the maps $F_s$ have a calibration property
which will be crucial to our result. However, essentially all of the difficulties in
the  proof of the main theorems are encountered in proving the following key
result whose proof we will postpone.

\begin{theorem}\label{thm:gradest}
  The gradient of the Jacobian of the natural map  exists off of the
  critical locus of $F_s$ and admits the following bound,
  $$\norm{\grad \op{Jac}(F_s)}\leq s^n\( C(n)(1+s)(1+\eta^{-n}) +
 5s\beta\),$$ where $C(n)$ is a constant depending only on $n$, $\eta$ is the
  injectivity radius of $(\til{M},\til{g})$ and  $\beta=\max\set{\rho^3,\norm{\grad \op{Rm}(g_o)}}$.
\end{theorem}

We will prove this theorem in Section \ref{sec:synthesis}. Using
this we can easily prove Theorem \ref{thm:critical}.

\begin{proof}[Proof of Theorem \ref{thm:critical}]
We begin by replacing the metric $g_o$ on $N$ with a nearby one with nearly
the same volume and curvature pinching, but with better derivatives of
curvature.  The main theorem of \cite{Bemelmans-Min-Oo-Ruh:84} employs a
Ricci flow theorem of Hamilton's to show that on the space of all closed
Riemannian manifolds $(N,g)$ with $-1\leq K(g)\leq 1$, the metrics can be
smoothed. Namely, there are uniform constants $T(n),c(n)$ and $c(n,m)$ and
metrics $g_\eps$ with Riemannian connection $\nabla_\eps$ such that
$$e^{-c(n)\eps}g\leq g_\eps\leq e^{c(n)\eps}g,\quad
\abs{\nabla-\nabla_\eps}\leq c(n)\eps, \quad \abs{\nabla_\eps^m
\op{Rm}(g_\eps)}\leq \frac{c(n,m)}{\eps^\frac{m}{2}}.$$ Moreover, this was
extended in Proposition 2.5 of \cite{Rong:96a} (whose proof Rong attributes to
T. Ilmanen and W.-X. Shi) to show there is a constant $c(n)$ such that $$\inf
K_g-c(n)\eps\leq K_{g_\eps}\leq \sup K_g +c(n)\eps.$$ (See also \cite{Shi:89}
and\cite{Kapovitch:05}.)

Applying this to the metric $\til{g}_o=\rho\cdot g_o$, with lower curvature
bound $-1$, we obtain $C^1$ close metrics $\til{g}_\eps$ with the listed
properties.

We now renormalize $\til{g}_\eps$ to the metric
$g_\eps=\(\frac{1}{\rho^2}-c(n)\eps\)\til{g}_\eps$ so that $g_\eps$ has
curvatures at most $-1$. The above controls imply, $\abs{u(g_\eps)-u(g_o)}\leq
c(n,\rho,\eps)$ and $\abs{\vol_{g_\eps}(N)-\vol_{g_o}(N)}<c(n,\rho,\eps)$ for
some constant $c(n,\rho,\eps)$ tending to 0 as $\eps\to 0$. In particular, since
$u(g_o)$ is bounded in terms of $\rho$ under the curvature assumption, we
have $u(g_o)\abs{\deg(f)}\vol_{g_o}(N)\leq
u(g_\eps)\abs{\deg(f)}\vol_{g_o}(N)
+c(n,\rho,\eps)\abs{\deg(f)}\vol_{g_o}(N)$. Hence for any $\delta>0$ there is
a sufficiently small $\eps$ depending only on
$n,\rho,\abs{\deg(f)}\vol_{g_o}(N)$ and $\delta$ such that
$u(g_o)\abs{\deg(f)}\vol_{g_o}(N)\leq u(g_\eps)\abs{\deg(f)}\vol_{g_o}(N)
+\delta.$ In other words, after decreasing the size of the constant $C$ in the
statement by a small uniform amount, we may assume the metric $g_o$ is such
that $\beta=\max\set{\rho^3,\norm{\grad \op{Rm}(g_o)}}$ is bounded by a
constant involving only $n, \rho, \abs{\deg(f)}\vol_{g_o}(N).$

If $\deg(f)=0$, then the inequality is trivially true. Hence, we will assume
$\deg(f)\neq 0$. By the gradient estimate of Theorem \ref{thm:gradest}, we
have that the Jacobian is at most $\norm{\grad \op{Jac}(F_s)}\ r$ on
$T_r(F_s)$. We want an $r$ such $\norm{\grad \op{Jac}(F_s)}\ r\leq s^n C
r\leq \frac{s^n}{2 u(g_o)^n}$. So we take $r=\frac{1}{2 C u(g_o)^n}$ where
$C=C(n,\eta,\beta)$. Note that $C$ may be treated as independent of $s$ since
$0\leq h(g)\leq n-1$ and we will choose $s$ sufficiently close to $h(g)$. Since
$n-1\leq u(g_o)\leq \rho(n-1)\leq \beta^{\frac13}(n-1)$, $r$ also depends
only on $n,\eta$ and the constant $\beta$.

We are assuming the metric $g_o$ has been smoothed, so that $\beta\leq C_2$
where $C_2=C_2(n,\rho,\abs{\deg(f)}\vol_{g_o}(N))$. Hence the radius $r$
depends only on $n,\rho,\eta$ and $\abs{\deg(f)}\vol_{g_o}(N)$. Finally, in
\cite{Gromov82} it is shown that under our curvature assumptions
$\frac{\vol_{g_o}(N)}{\norm{N}}$ is bounded above and below by constants
depending only on $n$ and $\rho$. (The generalized Gauss-Bonnet formula
shows that the proportionality of $\vol_{g_o}(N)$ to $\abs{\chi(M)}$ is
bounded above and below in the even dimensional case.)

On $T_r(F_s)$ we have the estimate,
$$\abs{\Jac F_s(x)}<\frac12\left(\frac{s}{u(g_o)}\right)^n.$$
Integrating, we have
\begin{align*}
  |\deg(f)|\vol_{g_o}({N})&=\left|\int_{{N}}\deg(f) dg_o\right|\\
 &\leq  \int_{M} |f^*dg_o| \\
 &\leq\int_M |\Jac F_s(y)|dg(y) \\
 &\leq\int_{M-T_r(F_s)} |\Jac F_s(y)|dg(y) +\int_{T_r(F_s)} |\Jac F_s(y)|dg(y)\\
 &< \left(\frac{s}{u(g_o)}\right)^n \(\vol_{g}(M)-\frac12\vol_{g}(T_r(F_s))\).
\end{align*}
Finally, take $s\to h(g)$ and multiply through by $u(g_o)^n$. Recall
that we have scaled $g$ so that $\kappa(g)=1$. If we scale a metric
by a constant $\frac{1}{c^2}>0$ then $\kappa\(\frac{ g}{c^2}\)=c\,
\kappa(g)$, $h\(\frac{g}{c^2}\)=c\, h(g)$ and
$\vol_{\frac{g}{c^2}}\(B_{\frac{g}{c^2}}\(p,\frac{r}{c}\)\)=c^{-n}\vol_{g}\(B_{g}(p,r)\)$.
Therefore scaling the metric back we obtain the given expression.

\end{proof}

Now we show that Theorems \ref{thm:ent-vol} and \ref{thm:vol} easily
follow.
\begin{proof}[Proof of Theorem \ref{thm:ent-vol}]
We use the previous theorem to obtain a condition under which $F_s$
can have no critical points. Suppose the critical locus of some
$F_s$, for $s$ very close to $h(g)$, is not empty. Then it contains
at least one point $p$. Hence if $C\leq \frac12 \vol_{g}(B(p,r))$,
then the inequality could not be satisfied. Hence $F_s$ would be a
local diffeomorphism, and in particular, a smooth $C^1$ cover.
Again, recall that we have scaled $g$ so that $\kappa(g)=1$.

Lastly, we recall a couple of standard results of differential
topology. Any $C^1$ structure on $M$ is $C^1$ equivalent to a $C^k$
structure for $k\in [1,\infty)\cup\set{\infty,\omega}$. Similarly,
any $C^1$ (local) diffeomorphism is homotopic to a $C^k$ (local)
diffeomorphism for any $k\in [1,\infty)\cup\set{\infty,\omega}$ (e.g
see Chapter 2 and Theorem 2.10 of \cite{Hirsch:76}). Hence, we
obtain the $C^\infty$ covering map, call it $F$, stated in the
conclusion which is homotopic to the original $C^0$ map $f$. Using
mollifiers, we can construct $F$ explicitly from $F_s$ and hence,
explicitly from $f$.
\end{proof}

\begin{remark}
If one were interested in the minimum regularity possible, then in
order to state the theorem note that we only need a $C^2$ structure
on $M$ and a $C^3$ structure on $N$. In this case, we obtain a $C^2$
covering map. Also, in the case when $N$ is negatively curved then
we obtain a $C^1$ limit map $\lim_{s\to h(g)} F_s$
(\cite{Besson-Courtois-Gallot:96}), however it is unlikely if this
exists when $M$ has mixed curvatures.
\end{remark}

\begin{proof}[Proof of Theorem \ref{thm:vol}]
For this we note that under the curvature assumptions, $ K_{g} \geq -1$ and $
K_{g_o} \leq -1$, we have $h(g)\leq n-1$ and $u(g_o)\geq n-1$. If $F_s$ has a
critical point at $p$, then by the proof of Theorem \ref{thm:ent-vol} we have
$h(g)^n\(\vol_{g}(M)-\frac12 \vol_{g}(B(p,r))\)\leq
u(g_o)^n\abs{\deg(f)}\vol_{g_o}(N)$. Therefore we obtain,
\begin{align*}
(n-1)^n\(\vol_{g}(M)-\frac12 \vol_{g}(B(p,r))\)&\geq h(g)^n
\(\vol_{g}(M)-\frac12\vol_{g}(B(p,r))\)\\
&\geq u(g_o)^n\abs{\deg(f)}\vol_{g_o}(N)\\
&\geq (n-1)^n\abs{\deg(f)}\vol_{g_o}(N).
\end{align*}
To finish, we note that by a classical result of Berger (see
\cite{Croke:80} for an improved constant), $\vol_{g}(B(p,r))\geq
C(n) r^n$ for all $r\leq \delta$ where $C(n)$ only depends on $n$
and $\delta$ is the injectivity radius of $M$.
\end{proof}
\section{Jacobian Estimates}\label{sec:Jac}

The barycenter of $\sigma_\lambda$ is defined to be the minimum of
the $C^1$-function $\CB_{\sigma_\lambda}(\cdot)$. In particular,
$\op{Bar}(\sigma_\lambda)=x$ if and only if the gradient of
$\mathcal{B}_{\sigma_\lambda}$ vanishes at $x$. This gradient can be
computed as follows
$$\nabla_x\mathcal{B}_{\sigma_\lambda}=\int_{\D\til{{N}}}\grad_x
B_\theta\ d\sigma_\lambda(\theta)=\int_{\til{{N}}}\int_{\D\til{{N}}}
\grad_x B_\theta\ d\nu_z(\theta)\  d\lambda(z),$$ where $\grad_x
B_\theta$ is the unit vector in $T_x\til{{N}}$ pointing to
$\theta\in\D\til{{N}}$. Applying this to $\lambda=\mu_y^s$, we have
$\sigma_\la=\sigma_y^s$ and the gradient vanishes at
$x=\til{F}_s(y)$. We denote by $r_z$ the function $r_z(x)=d(x,z)$.
Taking the covariant derivative of the gradient with respect to $y$,
i.e. directions $v\in T_y\til{M}$, yields
\begin{align*}
0=D_v\nabla_{F_s(y)}\mathcal{B}_{\sigma_y^s}=\int_{\D\til{{N}}}
&D_{d_y F_s(v)}\grad B_\theta\  d\sigma_y^s(\theta)\\
&-s\int_{\til{M}}\int_{\D\til{{N}}} \grad_{ F_s(y) } B_\theta\tensor
d_y r_z(v)\ d\nu_{\tilde{f}(z)}(\theta)\, d\mu_y^s(z).
\end{align*}

Therefore we have,
\begin{align*}
d_yF_s&=s\(\int_{\D\til{{N}}} Dd_{F_s(y)}B_\theta d\sigma_y^s\)^{-1}
\int_{\til{M}}\(\int_{\D\til{{N}}} \grad_{ F_s(y) } B_\theta \
d\nu_{\tilde{f}(z)}(\theta)\)\tensor d_y r_z(y)
 \, d\mu_x^s(z),
\end{align*}
where $Dd_{F_s(y)}B_\theta$ is the (1,1)-tensor associated to the
Hessian of $B_\theta$ at the point $F_s(y)$. More specifically, it
is the self adjoint linear map from $T_{F_s(y)}\til{{N}}\to
T_{F_s(y)}\til{{N}}$ such that
$Dd_{F_s(y)}B_\theta(\grad_{F_s(y)}B_\theta)=0$ and
$Dd_{F_s(y)}B_\theta$ restricted to
$\(\grad_{F_s(y)}B_\theta\)^\perp$ is the second fundamental form of
the horosphere through $F_s(y)$ and tangent to $\theta$.

We can rewrite the previous expression more concisely as
\begin{align*}
d_yF_s=s\underbrace{\left(\int_{\D \til{{N}}} DdB_{\theta}\
d\sigma_y^s\right)^{-1}}_{A^{-1}}\underbrace{\int_{\widetilde{M}}\int_{\D
\til{{N}}} \grad B_\theta \tensor d r_z
\,d\nu_{\tilde{f}(z)}(\theta)\, d\mu_y^s(z)}_{H}.
\end{align*}
For any $v\in T_{y}\til{M}$ we have, assuming the directional
derivatives exist,
\begin{align}
\begin{split}\label{eq:Jacob}
\grad_v\Jac{F_s}&=s^n\grad_v\frac{\det H}{\det A} \\
&=s^n \frac{\det H}{\det A}\left( \tr(\grad_v H\,H^{-1})-\tr(\grad_v
A\ A^{-1})\right)\\
&=s^n\underbrace{\frac{\det H}{\det A}}_{\circled{1}}\left(
\underbrace{\tr( H^{-1}\grad_v H)}_{\circled{2}}-\underbrace{\tr(
A^{-1}\grad_v A)}_{\circled{3}}\right)
\end{split}
\end{align}
where the traces and determinants are with respect to the metrics $g_o$ on
$T_{F_s(y)}\til{N}$ and $g$ on $T_y\til{M}$.

We can compute the derivative terms as,
\begin{align*}
\grad_v H&=\Biggr(\int_{\widetilde{M}}\int_{\D
\til{{N}}}\Big[D_{d_yF_s(v)}dB_{\theta}\tensor dr_z+\grad
B_{\theta}\tensor
Dd r_z(v)\\
&\hspace{4.5cm}-s\grad B_{\theta}\tensor dr_z\ \inner{\grad_yr_z,v}\Big]\ d\nu_{\tilde{f}(z)}(\theta)\, d\mu_y^s(z)\Biggr)\\
\intertext{and}& \\
\grad_v A&=\int_{\widetilde{M}}\int_{\D \til{{N}}}
\Big[D_{d_yF_s(v)}DdB_{\theta}\ -s\
DdB_{\theta}\inner{\grad_yr_z,v}\Big]\ d\nu_{\tilde{f}(z)}(\theta)\,
d\mu_y^s(z).
\end{align*}

The existence of $\grad_v\op{Jac}(F)$ will follow from the continuity of the
terms, assuming they can be bounded. Except where otherwise specified, for
the remainder of  the paper $\norm{A}$ will represent the operator norm
(largest singular value) on the tensor $A$ induced from the metric norm on
tangent vectors and cotangent vectors. For a measure $\nu$, the quantity
$\norm{\nu}$ is its total mass. We will concentrate on the estimates of
$\circled{2}$ and $\circled{3}$ in terms of $\circled{1}$ for the remainder of
this section.

Recall that  $\eta$ is the injectivity radius of the universal cover of $M$. Note
that $\eta=\injrad(\til{g})\geq \injrad{g}=\delta$.
\begin{prop}\label{prop:gradH}
We have $\norm{\grad H}\leq C(1+s)\(1+\eta^{-n}\)\norm{\mu_y^s}$ for a
constant $C$ depending only on $n$.
\end{prop}
Before proving this, we will need a lemma.

\begin{lemma}
We have $$\int_{\til{M}} \abs{\Delta_y r_z}d\mu^s_y(z)\leq
C(1+s)\(1+\eta^{-n}\) \norm{\mu_y^s},$$ where $C$ is a constant
depending only on $n$ and $\eta$ is the injectivity radius of
$\til{M}$.
\end{lemma}
\begin{proof}
Set $r_z(y)=d(y,z)$. Since spherical Jacobi tensors satisfy the Sturm-Liouville
equation, we have $\Delta_y d(y,z)=\Delta_z d(y,z)$ (see
\cite{Eschenburg-Heintze90}/. So we estimate $\int_{\til{M}} \abs{\Delta_z
r_y}d\mu^s_y(z)$. If $\hat{n}_{\Omega}$ is the outward pointing normal to a
domain $\Omega\subset \til{M}$ with smooth boundary, then we have by
Stokes Theorem,
\begin{align*}
\int_{\D \Omega}\inner{\grad e^{-s r_y},\hat{n}_\Omega}dg_{\D\Omega}&=
\int_{\Omega}
\Delta e^{-sr_y} dg\\
&=\int_\Omega \op{div}(-s e^{-sr_y}\grad r_y)dg\\
&= \int_\Omega -s e^{-sr_y}\Delta r_y +\inner{\grad -s
e^{-sr_y},\grad r_y}dg\\
&=\int_\Omega -s e^{-sr_y}\Delta r_y +s^2 e^{-s r_y}\inner{\grad
r_y,\grad r_y}dg.
\end{align*}
 Since
$\abs{\grad r_y}=1$, we obtain
\begin{align*}
\int_{\Omega} \Delta r_y\  e^{-sr_y}dg&=s\int_{\Omega} e^{-sr_y}dg
-\frac1s\int_{\D
\Omega}\inner{\grad e^{-s r_y},\hat{n}_\Omega}dg_{\pa \Omega} \\
&=s\int_{\Omega} e^{-sr_y}dg +\int_{\D \Omega} e^{-sr_y}\inner{\grad
r_y,\hat{n}_\Omega}dg_{\pa \Omega}.
\end{align*}

Let $\op{Cut}_y\subset \til{M}$ denote the cut locus from $y$, and for any
$r>0$,  set $V(r)=\exp_y^{-1}(B(y,r)-\op{Cut}_y)$. The set $V(r)\subset
T_y\til{M}$ is star convex from $0$. By retracting each ray from $y$ to the
points in $\pa V(r)$ by an appropriate amount between $0$ and $\eps$ we can
obtain a domain $V_\eps(r)\subset V(r)$ with smooth boundary. By
construction $V_\eps(r)$ is star convex from $0$ and since $\exp_y$ is a
diffeomorphism on $V(r)$, the image $\Omega_\eps =\exp_y V_\eps(r)$ will
have smooth boundary and is star convex from $y$. If $\hat{n}_\eps$ is the
outward pointing normal field to $\pa \Omega_\eps$ then the star convexity of
$V_\eps(r)$ implies that $\inner{\grad_z r_y,\hat{n}_\eps(z)}\geq 0$ for all
$z\in \pa \Omega_\eps$. Applying the previous estimate to
$\Omega=\Omega_\eps$ we obtain,

\begin{align*}
\int_{\Omega_\eps} \Delta r_y\  e^{-sr_y}dg &=s\int_{\Omega_\eps}
e^{-sr_y}dg +\int_{\pa \Omega_\eps}e^{-s r} \inner{\grad
r_y,\hat{n}_\Omega}dg_{\pa \Omega}\\
&\geq  s\int_{\Omega_\eps} e^{-sr_y}dg.
\end{align*}

Taking $\eps \to 0$ we obtain that $\int_{B(y,r)-\op{Cut}_y} \Delta r_y\
e^{-sr_y}\ dg\geq  s\int_{B(y,r)-\op{Cut}_y} e^{-sr_y}dg.$ In fact $\Delta r_y
dg$ extends to the cut locus as well as a signed measure (or see the bottom of
p.257 in\cite{Petersen98} for how to define $\Delta r_y$ using support
functions). Since the cut locus has $0$ measure, and the Laplacian is controlled
there, we can ignore it and simply write,
$$\int_{B(y,r)} \Delta
r_y\ e^{-sr_y}\ dg\geq  s\int_{B(y,r)} e^{-sr_y}dg.$$ Finally,
taking $r\to\infty$ we obtain,
$$\int_{\til{M}}\Delta r_y\ d\mu_y^s\geq  s \norm{\mu_y^s}.$$

Set $\Delta^+ r_y=\max\set{\Delta r_y,0}$ and $\Delta^- r_y=\min\set{\Delta
r_y,0}$. These are defined off of the cut locus of $y$ in $\til{M}$. Recall that this
has measure $0$. The lower curvature bound on $M$ implies by a standard
Ricci comparison result (see e.g. \cite{Chavel93}) that off of the cut locus of
$\til{M}$ from $y$, the mean curvature of distance spheres, i.e. the trace of the
shape operator, of $S(y,r_y(z))$ is less than $(n-1)\coth(r_y(z))$. In other
words, $\Delta^+ r_y\, (z)\leq (n-1)\coth(r_y(z))$. Therefore we have
\begin{align*}
\int_{\til{M}} \abs{\Delta r_y}d\mu^s_y(z)&=\int_{\til{M}} \Delta^+
r_y \ d\mu^s_y(z)-\int_{\til{M}}  \Delta^- r_y \ d\mu^s_y(z)\\
&=2\int_{\til{M}} \Delta^+ r_y \ d\mu^s_y(z)-\int_{\til{M}} \Delta
r_y \ d\mu^s_y(z)\\
&\leq 2(n-1)\int_{\til{M}} \coth(r_y) \
d\mu^s_y(z)-s\norm{\mu_y^s}\\
\end{align*}

On the other hand, for the restricted metric $g'$ on $S(y,t)$ we may write
$\vol_{g'}(S(y,t))=\int_{S_y\til{M}}\dvol(v,t) dv$ in terms of the radial spherical
volume element $\dvol$. From the curvature assumptions we have
$\dvol(v,t)\leq \(\sinh(t)\)^{n-1}$.

Hence we have
$$\int_{B\(y,\frac{1}{2+s}\)}\coth(r_y)d\mu^s_y(z)\leq
\omega_{n-1}\int_{0}^{\frac{1}{2+s}} \cosh(t)\sinh^{n-2}(t)e^{-s t}\
dt\leq \frac{\omega_{n-1}}{(1+s)^{n-1}}.$$ For the last inequality
we used the fact that $\sinh(t)\leq \frac{t e^t}{1+t}$ and that
$\cosh(t)<2$ on the interval.

Since $\coth(t)\leq 1+\frac{1}{t}$ for all $t>0$, we have
$$\int_{\til{M}\setminus B\(y,\frac{1}{2+s}\)}\coth(r_y)d\mu^s_y(z)\leq
(3+s)\int_{\til{M}\setminus B\(y,\frac{1}{2+s}\)}d\mu^s_y(z)\leq
(3+s)\norm{\mu_y^s}.$$ Hence,

\begin{align}\label{eq:almost}
\int_{\til{M}} \abs{\Delta r_y}d\mu^s_y(z)\leq 2(n-1)\((3+s)
\norm{\mu^s_y}+\frac{\omega_{n-1}}{(1+s)^{n-1}}\)-s\norm{\mu_y^s}.
\end{align}

For later use it will be essential that this bound scale linearly in
$\norm{\mu_y^s}$ as in the statement of the lemma. Hence we must bound the
$\frac{\omega_{n-1}}{(1+s)^{n-1}}$ in terms of the size of this measure.

By a result of Berger (see \cite{Croke:80},\cite{Croke:88} for a stronger
version) there is a constant $C_1(n)\geq \frac{\omega_n}{\pi^n}$ depending
only on $n$, such that $\vol_{g}(B(x,r))\geq C_1(n)r^n$ for all $r<\eta$. We
then have,

\begin{align*}
\norm{\mu_y^s}&=\int_0^\infty e^{-st}\vol^{\prime}(B(y,t))\ dt\\
&\geq \int_0^\eta e^{-st}\vol^{\prime}(B(y,t))\ dt\\
&=e^{-s\eta}\vol_{g}(B(y,\eta))+s\int_0^\eta
e^{-st}\vol_{g}(B(y,t))\ dt\\
&\geq e^{-s\eta}\vol_{g}(B(y,\eta))+s C_1(n)\int_0^\eta e^{-st}t^n \
dt\\ &=e^{-s\eta}\vol_{g}(B(y,\eta))+s
C_1(n)\frac{\min\set{\eta^{1+n},(n+1)!}}{(1+n)(1+s)^{n+1}}\\
& \geq C_2(n)\frac{\eta^n}{(1+s)^n}
\end{align*}
where the first equality is by expressing the volume of the sphere as the
derivative of the volume of the ball. The next inequality is just restricting the
integral to the finite domain $[0,\eta]$. The next equality is integration by parts.
The next inequality follows from the estimate for the volume of balls up to the
injectivity radius mentioned above. The integral is then evaluated by noting that
the function $(1+s)^n\int_0^\eta e^{-st}t^n \ dt$ is either monotone or
unimodal in $s$ depending on $\eta$, and hence it is bounded below by its
limits as $s\to 0$ or $s\to \infty$. The last inequality follows from noting that
the first term is larger than $\frac{1}{3}C_1(n)\eta^n$ for $s\leq
\frac{1}{\eta}$ and then choosing a sufficiently small constant $C_2(n)$.

In particular, $$\frac{\omega_{n-1}}{(1+s)^{n-1}}\leq (1+s)C_3(n)\eta^{-n}
\norm{\mu_y^s}$$ for some constant $C_3(n)$ depending only on $n$. Putting
this together with the estimate \eqref{eq:almost} and choosing
$C(n)=\max\set{6n-6,(2n-2)C_3(n)}$ gives the lemma.

\end{proof}

\begin{remark}
We could have shortened the end of the previous lemma slightly if
the lemma just asked for an unspecified bound in terms of $s$ and
the injectivity radius of $\til{M}$. In fact, the linear dependence
on $s$ and $\eta$ is optimal up to constants. To see this, we can
take a manifold with $h(g)=1$ and with y at the tip of a long spike
with injectivity radius $\eta$.

This estimate is the only term which does not have $s$ as a factor.
This is in fact necessary since if $M$ has constant curvature $-1$,
then $\norm{\int_{\til{M}} Ddr_z d\mu_y^s}> \norm{\mu_y^s}$.
\end{remark}

We now establish a couple of nontrivial  properties of the operator norm which
we will need later.
\begin{lemma}\label{lem:norm1}
For any square matrix $C\in M_{nxn}$ and symmetric matrices $A\in M_{nxn}$
and $B\in M_{nxn}$ satisfying $\norm{A(v)}\leq\norm{B(v)}$ for all
$v\in\R^n$ we have
$$\norm{AC}\leq \norm{BC}\quad\text{and}\quad
\norm{CA}\leq\norm{CB}.$$
\end{lemma}
\begin{proof}
We may assume $B$ is invertible, otherwise we take limits in $\op{GL}_n$.
Since $\inner{Av,Av}\leq \inner{Bv,Bv}$ by replacing $v$ with $B^{-1}v$ we
have $\inner{AB^{-1}v,AB^{-1}v}\leq \inner{v,v}$. Then $\norm{v}\leq
\norm{AB^{-1}v}=\norm{B^{*-1}A^*v}=\norm{B^{-1}Av}$. Setting $w$ to be
the unit vector $w=\frac{B^{-1}Av}{\norm{B^{-1}Av}}$ for any unit vector $v$
we obtain
$$\norm{Av}\leq \frac{\norm{Av}}{\norm{B^{-1}Av}}=\frac{\norm{BB^{-1}Av}}{\norm{B^{-1}Av}}=\norm{Bw}.$$
On the other hand $Av=Bw\frac{\norm{Av}}{\norm{Bw}}$. So taking $v$ to be
a unit vector with $\norm{CAv}=\norm{CA}$, we obtain
$$\norm{CA}=\norm{CAv}=\norm{CBw}\frac{\norm{Av}}{\norm{Bw}}\leq \norm{CBw}\leq \norm{CB}.$$
Since $\norm{CA}=\norm{A^*C^*}=\norm{AC}$, we obtain
$\norm{AC^*}\leq\norm{BC^*}$ which gives the second result since $C$ was
arbitrary.
\end{proof}

\begin{lemma}\label{lem:norm2}
For any matrix $A$, positive definite matrices $B_i$, and numbers $\rho_i$ with
$\abs{\rho_i}\leq C$ for $i=1,\dots, k$, we have
$$\norm{A\(\sum_{i=1}^k \rho_i
B_i\)}\leq 3C\norm{A\(\sum_{i=1}^k B_i\)}\quad\text{and}\quad\norm{\(\sum_{i=1}^k \rho_i
B_i\)A}\leq 3C\norm{\(\sum_{i=1}^k B_i\)A}.$$
\end{lemma}

\begin{proof}
Set $B_\rho=\sum_{i=1}^k \rho_i B_i$ and $B=\sum_{i=1}^k B_i$. First assume
that each $\rho_i>0$. For any vector  $v$, we have $v.B_i.v\geq 0$ and so
$$\inner{v,B_\rho( v)}=\sum_{i=1}^k \rho_i  \inner{v,B_i(v)}\leq C\sum_{i=1}^k
\inner{v,B_i(v)}=C\inner{v,B(v)}.$$ In particular, this is true for all
eigenvectors, so $\inner{B_\rho(v),B_\rho(v)}\leq C^2\inner{B(v),B(v)}$ for
all vectors. Hence for some choice of unit $v$ we have
$$\norm{B_\rho A}^2=\inner{B_\rho A(v),B_\rho A(v)}\leq C^2\inner{BA(v),BA(v)}=C^2\norm{BA(v)}^2\leq C^2\norm{BA}^2.$$
For the case when $\rho_i$ may be negative we have $\norm{B_\rho A}=
\norm{B_{\rho+C} A+C BA}\leq \norm{B_{\rho+C}A}+C \norm{BA}\leq
2C\norm{BA}+\norm{BA}$ by the previous case. Transpose invariance yields
the  result with $A$ and $B$ reversed.
\end{proof}

\begin{proof}[Proof of Proposition \ref{prop:gradH}]
We treat each of the three terms separately. First, since $\abs{dr_z(v)}\leq 1$
for any $v$ and $Dd_{F_s(y)}B_\theta$ is positive semi-definite and symmetric,
by Lemma \ref{lem:norm2} we have,

\begin{align*}
\biggr\|\int_{\til{M}}\int_{\pa \til{N}} Dd_{F_s(y)}B_\theta\of
\, & dF_s(y)\tensor dr_z(v) \ d\nu_{\tilde{f}(z)}(\theta)\,d\mu_y^s(z)\biggr\|\\
&=\norm{\(\int_{\til{M}}dr_z(v)\int_{\pa \til{N}} Dd_{F_s(y)}B_\theta\
d\nu_{\tilde{f}(z)}(\theta)\,d\mu_y^s(z)\)\of dF_s(y)}\\
&\leq 3\norm{\(\int_{\D \til{{N}}} DdB_{\theta}\
d\sigma_y^s(\theta)\)\of dF_s(y)}\\
&=3\norm{sA A^{-1} H}=3s\norm{H}.
\end{align*}

The singular values of this quantity are all bounded by $3s\norm{\mu_y^s}$.

By the previous lemma we have a bound on $\int_{\til{M}}\abs{\tr
Ddr_z} d\mu_y^s$. However, the lower curvature bound of $-1$  on
$\til{M}$ implies that all eigenvalues of $Ddr_z$ are less than
$\coth(r_z)$. So as in the proof of the lemma, we have that the
eigenvalues of $\int_{\til{M}}Ddr_z d\mu_y^s$ are bounded from above
by $[(3+s)+(1+s)C_3 \eta^{-n}]\norm{\mu_y^s}$. Combining this with
the bound from the lemma of the integral of the absolute value of
the trace, we obtain a bound for how negative these eigenvalues can
be. Putting these together, we obtain that
$$\int_{\til{M}} \norm{Ddr_z} \,d\mu_y^s(z)<2C(1+s)\(1+\eta^{-n}\)\norm{\mu_y^s}$$
where $C$ is a constant depending only on $n$.

For the final term we have $\norm{\grad B_\theta}=1=\norm{dr_z}$.
Hence the final term has singular values bounded by
$\norm{\mu_y^s}$.

The singular values of the sum of tensors is bounded by the sum of their
singular values. Therefore after absorbing an extra factor of $6$ into our
constant $C$, we have the given universal bound on the tensor
$\frac{1}{\norm{\mu_y^s}}\norm{\grad H}$.
\end{proof}

We now proceed to the second main estimate. Recall that we set
$\beta=\max\{\rho^3,\norm{\nabla Rm}\}$ to be the bound on  the
derivatives up to first order of the curvature tensor of $N$.
\begin{prop}\label{prop:Apart}
We have $\norm{A^{-1}\grad A}_{\infty}\leq 4\beta
s\norm{\mu_y^s}\norm{A^{-1}}+s$.
\end{prop}

Again we will prove this proposition by dealing with each of the terms
separately. However, first we must deal with the regularity of a fixed
horosphere.

The strong unstable foliation $W^{su}$ for the geodesic flow on $SN$ is in
general only H\"{o}lder continuous  whenever $\rho>2$. On the other hand, it
follows from a version of the Hadamard-Perron Theorem (or see Theorem 8 of
\cite{Anosov:69}) that the leaves of this foliation are individually $C^\infty$. In
particular, on a closed manifold of negative curvature, the horospheres
$B_\theta^{-1}(0)$ are $C^\infty$ submanifolds for fixed $\theta$ (see also
\cite{Heintze-Hof77}).  It is a fairly well-known result, e.g Remark 3.3 in Chapter
4 of \cite{Ballmann:95}, that for fixed $\theta$, $\nabla_vDdB_\theta$ depends
on $\norm{\nabla Rm}$ and $\rho$. However, we are not aware of any explicit
estimates to this effect. Therefore, the next lemma makes this dependency
precise.

\begin{lemma}\label{lem:DDdB}
For any $v\in SN$ and $\theta\in \partial \til{N}$, we have,  $\norm{\nabla_v
DdB_\theta} \leq 2 \beta$.
\end{lemma}

\begin{proof}
Recall that for fixed $\theta$ the symmetric tensor $DdB_\theta$ is $0$ in the
direction $\grad B_\theta$. In particular, for all $u\in SN$,
$(\grad_uDdB_\theta)(\grad B_\theta)=\grad_u(DdB_\theta(\grad
B_\theta))-DdB_\theta(\grad_u\grad B_\theta)=-DdB_\theta^2(u)$.

Let $w$ represent the geodesic vector field defined at each point $z\in \til{N}$
by $w(z)=\grad_z B_\theta$. In other words, $-w(z)$ is the unique unit vector
at $z$ pointing toward $\theta$. We indicate the bounds in terms of the
curvature tensor. The Ricatti equation gives,
$$\grad_{w} DdB_\theta
+(DdB_\theta)^2+R(w,\cdot,w)=0.$$ Observe that the vector field $w$
can be viewed as a stable submanifold of $S\til{N}$. Choose an extension of $v$
which is an unstable Jacobi field along the flow lines of $w$. Then $[w,v]=0$
and so
$$R(w,v,\cdot)=\nabla_{[w,v]}-[\grad_w,\grad_v]=\grad_v\grad_w-\grad_w\grad_v.$$
Covariantly differentiating this with respect to $v$, and applying the Ricatti
identity we obtain,
\begin{align*}
  0&=\nabla_v\nabla_w DdB_\theta +\nabla_v(DdB_\theta)^2
  +\nabla_vR(w,\cdot,w)\\
  &=\nabla_w\nabla_v DdB_\theta +R(w,v,\cdot)DdB_\theta
  +\nabla_v(DdB_\theta)\ DdB_\theta+ DdB_\theta\ \nabla_v(DdB_\theta)
  +\nabla_vR(w,\cdot,w).
\end{align*}
In other words,
\begin{align}\label{eq:diffeqU}
\nabla_w\nabla_v DdB_\theta=-DdB_\theta
\nabla_v(DdB_\theta)-\nabla_v(DdB_\theta)\,DdB_\theta -R(w,v,\cdot)
DdB_\theta-\nabla_vR(w,\cdot,w).
\end{align}
Recall that for any symmetric 2-tensor $U$, $\nabla_v U$ is symmetric since
$(\nabla_vU)(X,Y)=\nabla_v(U(X,Y))-U(\nabla_vX,Y)-U(X,\nabla_vY).$ Also
observe that $R(w,v,\cdot)$ is skew symmetric, but that $R(w,v,\cdot)
DdB_\theta$ is symmetric since $DdB_\theta
\nabla_v(DdB_\theta)+\nabla_v(DdB_\theta)\,DdB_\theta$ is symmetric along
with the remaining terms of equation \eqref{eq:diffeqU}. By assumption the
eigenvalues of $DdB_\theta$ are between $1$ and $\rho$, and the eigenvalues
of $\nabla_vR(w,\cdot,w)$ are between $-\beta$ and $\beta$. Similarly the
eigenvalues of $R(w,v,\cdot)DdB_\theta$ are between $-\rho^3$ and
$\rho^3$. Hence, the differential equation dictates that if an eigenvalue of
$\nabla_v(DdB_\theta)$ is larger than $\beta$ then the corresponding unit
eigenvector $u$ must be perpendicular to $w$ and satisfies
$\inner{u,\nabla_w\nabla_v DdB_\theta(u)}<0.$

Since $v$ is a Jacobi field along the geodesic direction $w$, $[w,v]=0$ and so
$\nabla_vw=\nabla_w v=DdB_\theta(v)$. Therefore for any field $X$,
$(\nabla_vDdB_\theta)(w,X)=\nabla_v(DdB_\theta(w,X))-DdB_\theta(\nabla_v
w,X)-DdB_\theta(w,\nabla_vX)=\inner{-DdB_\theta^2(v),X}$ or simply
$\nabla_v(DdB_\theta)(w)=-DdB_\theta^2(v)$. Similarly,
$(\nabla_vR(w,\cdot,w))(w)=-R(w,DdB_\theta(v),w)$, which we observe is also
orthogonal to $w$. Let $u$ be the unit vector field along the geodesic tangent to
$w$ such that $u$ is the eigenvector  for the maximal eigenvalue $q(v)$ of
$\nabla_vDdB_\theta$ at each point of the geodesic.  Since $u$ is a unit field,
$<\nabla_w u,u>=0$ and so we have
\begin{align*}
\nabla_w(\inner{u,\nabla_vDdB_\theta(u)})&=\inner{\nabla_w u,\nabla_vDdB_\theta(u)}+\inner{ u,\nabla_w(\nabla_vDdB_\theta(u))}\\
&=\inner{\nabla_w u,q(0) u}+\inner{ u,(\nabla_w\nabla_vDdB_\theta)(u)+\nabla_vDdB_\theta(\nabla_w u)}\\
&=0+\inner{ u,(\nabla_w\nabla_vDdB_\theta)(u)}+\inner{\nabla_vDdB_\theta(u),\nabla_wu}\\
&=\inner{ u,(\nabla_w\nabla_vDdB_\theta)(u)}.
\end{align*}
Now we express $u=aX+b w$ for a unit vector $X$ with $\inner{X,w}=0$ and
functions $a$ and $b$ along the geodesic tangent to $w$ satisfying
$a^2+b^2=1$. Since $\nabla_vDdB_\theta(w)=-DdB_\theta^2(v)$, we have
$$q(v)b=\inner{q(v)u,w}=\inner{\nabla_vDdB_\theta(u),w}=\inner{aX+bw,\nabla_vDdB_\theta(w)}=-a DdB_\theta^2(v,X).$$
Since $-DdB_\theta^2(v)$ has eigenvalues with norm at most $\rho^2$,  if
$q(v)>\rho^2$ then $\abs{b}<\abs{a}$. Evaluating from the earlier formula,
we obtain,
 \begin{align*}
\nabla_w q(v)&=\inner{u,(\nabla_w\nabla_v DdB_\theta)(u)}\\
 &=-2\inner{\nabla_v
 DdB_\theta(u),DdB_\theta(u)}-\inner{R(v,w,u),DdB_\theta(u)}-\inner{(\nabla_vR(w,\cdot,w))(u),u}\\
 &= -2 q(v)\inner{u,DdB_\theta(u)}  - a^2\inner{R(v,w,X),DdB_\theta(X)}-a\inner{(\nabla_vR(w,\cdot,w))(X),u}+b\inner{R(w,DdB_\theta(v),w),u}\\
 &=-2q(v)a^2\inner{X,DdB_\theta(X)}-a^2\inner{R(v,w,X),DdB_\theta(X)}-a^2\inner{(\nabla_vR(w,\cdot,w))(X),X}+2ab \inner{R(w,v,w),DdB_\theta(X)}\\
 &\leq -2a^2q(v)+a^2\rho^3+a^2\beta+2\abs{ab}\rho^3< a^2(-2q(v)+4\beta).
 \end{align*}

Here we have used that for all unit vectors $a,b,c,d$ at any point,
$\abs{R(a,b,c,d)}\leq \rho^2$ (e.g. see Lemma 3.7 of
\cite{Bourguignon-Karcher:78}). Hence if $q(v)>2 \beta$ then $\nabla_w
q(v)<0$. Now observe that $\nabla_w
q\(\frac{v}{\abs{v}}\)=\frac{1}{\abs{v}}\nabla_w q(v)+q(v)\nabla_w
\frac{1}{\abs{v}}$. Since the vector $v$ was extended as an unstable Jacobi
field, $\nabla_w \frac{1}{\abs{v}}<0$. Hence if $q(v)>2 \beta$, then $\nabla_w
q\(\frac{v}{\abs{v}}\)<0$.

Since the diagonal action of $\pi_1(N)$ on $S\til{N}\times\pa \til{N}$ is
cocompact and $q(v)$ is continuous, there is a $v\in S\til{N}$ and a $\theta\in
\pa \til{N}$ where $q(v)$ achieves a maximum. This maximum cannot be larger
that $2 \beta$. Otherwise for the extension of $v$ as an unstable unstable
Jacobi field along the geodesic through the vector $w$ pointing to $\theta$,
$q\(\frac{v}{\abs{v}}\)$ is strictly increasing in the $-w$ direction,
contradicting that $v,\theta$ was a maximum for $q$.

Similarly, let $q(v)$ be the function giving the minimal eigenvalue for $\nabla_v
DdB_\theta$ along $w$, and $u$ is the corresponding eigenvector field. If
$q(v)<-2\beta$, then the continuing from the second to last line in the
computation of $\nabla_w q(v)$ we obtain
$$\nabla_w q(v)\geq -2a^2q(v)
-a^2\rho^3-a^2\beta-2\abs{ab}\rho^3>-a^2(2 q(v)+4\beta)>0.$$ Moreover,
$\nabla_w q\(\frac{v}{\abs{v}}\)=\frac{1}{\abs{v}}\nabla_w
q(v)+q(v)\nabla_w \frac{1}{\abs{v}}>0$. Analogously to the maximum case,
the minimum of $q(v)$ over all $v\in SN$ and $\theta\in \pa\til{N}$ cannot be
less than $-2\beta$ since then it is increasing in the stable direction.
\end{proof}

\begin{lemma}
For all $v\in S_y\til{M}$, we have the uniform bound
$$\norm{A^{-1}\int_{\pa \til{N}}D_{d_yF_s(v)}DdB_\theta
d\sigma_y^s}\leq 4\beta s\norm{\mu_y^s}\norm{A^{-1}}.$$
\end{lemma}

\begin{proof}
For each $\theta\in \pa \til{N}$,  let $w_\theta$ represent the geodesic vector
field defined at each point $z\in \til{N}$ by $w_\theta(z)=\grad_z B_\theta$. Fix
$u\in S_y\til{M}$. For all $v\in S_{F_s(y)}\til{M}$, we may write
$v=a(\theta)X_\theta+b(\theta) w_\theta$ for a unit vector $X_\theta$ with
$\inner{X_\theta,w_\theta}=0$, $b(\theta)=\inner{w_\theta,v}$ and
$a(\theta)=\sqrt{1-\inner{w_\theta,v}^2}$. Since $\nabla_u
DdB_\theta(w_\theta)=-DdB_\theta^2(u)$ and
$a(\theta)^2=1-\inner{w_\theta,v}^2$, we have
$$\nabla_{d_yF_s(u)}DdB_\theta(v,v)=(1-\inner{w_\theta,v}^2)
\nabla_{d_yF_s(u)}DdB_\theta(X_\theta,X_\theta)-2a(\theta)b(\theta)DdB_\theta^2(d_yF_s(u)).$$
Recall  that in our notation, $d_yF_s=sA^{-1}H$. Since $a(\theta)$ and
$b(\theta)$ are at most $1$ and $\norm{Ddb_\theta}\leq \rho$, integrating the
above expression gives
\begin{align*} \int_{\pa
\til{N}}a(\theta)b(\theta)DdB_\theta^2(d_yF_s(u),v)\,d\sigma_y^s(\theta)&\leq \rho \(\int_{\pa
\til{N}}DdB_\theta^2(d_yF_s(u))\,d\sigma_y^s(\theta)\)(sA^{-1}H(u),v)\\
&=s\rho\inner{AA^{-1}H(u),v}=s\rho\inner{H(u),v}.
\end{align*}

Putting this together with the estimate from Lemma \ref{lem:DDdB}, we obtain
$$\abs{\int_{\pa \til{N}}D_{d_yF_s(u)}DdB_\theta(v,v) d\sigma_y^s(\theta)}\leq
\norm{d_yF_s(u)}2\beta\(\int_{\pa \til{N}}1-\inner{w_\theta,v}^2
d\sigma_y^s(\theta)\)+2s\rho\inner{H(u),v}.$$ On the other hand, and because of our
normalization $ K_{g_o} \leq -1$, we have
$$A(v,v)=\int_{\pa \til{N}}DdB_\theta(v,v) d\sigma_y^s(\theta)\geq
\int_{\pa \til{N}}1-\inner{w_\theta,v}^2
d\sigma_y^s(\theta)=\norm{\mu_y^s}(\Id-Q)(v,v),$$ where
$Q=\frac{1}{\norm{\mu_y^s}}\int_{\pa \til{N}}w_\theta\tensor (w_\theta)^*
d\sigma_y^s(\theta)$. Now applying Lemma \ref{lem:norm1} twice, using
these estimates on each factor, we have
\begin{align*}
\norm{A^{-1}\int_{\pa \til{N}}D_{d_yF_s(u)}DdB_\theta d\sigma_y^s(\theta)}&\leq
\norm{\(I-Q\)^{-1}\(I-Q\)}2\beta\norm{d_yF_s(u)}+2s\rho\norm{A^{-1}H(u)}\\
&=\(2\beta+2\rho\)\norm{d_yF_s(u)}\leq 4\beta\norm{d_yF_s(u)}.
\end{align*}
To complete the proof, we note that $\norm{H}\leq \norm{\mu_y^s}.$

\end{proof}

\begin{proof}[Proof of Proposition \ref{prop:Apart}]
The first term of the estimate is controlled by the previous Lemma. Noting that
$\grad r_z$ has unit length where it is defined off of the cut locus of measure
$0$, the second term is estimated  by,
$$\norm{A^{-1}\(\int_{\til{M}}\int_{\pa \til{N}} s
Dd_{F_s(y)}B_\theta\ \inner{\grad_y r_z,v} d\nu_{\tilde{f}(z)}(\theta)\,d\mu_y^s(z)\)}\leq
\norm{sA^{-1}A}\leq s.$$
\end{proof}

\section{Synthesis of the estimates}\label{sec:synthesis}

Here we treat the second part of the main estimate. In formula
\eqref{eq:Jacob}, the term $\circled{1}$ does not have a uniform upper bound.
However, we can bound each of the two terms $\circled{2}$ and $\circled{3}$
in term of the reciprocal of $\circled{1}$.

\subsection{Estimate of $\circled{1}$}

For any $v\in S_{F_s(y)}\til{{N}}$ set
$$t_v(\theta)=\angle_{F_s(y)}\( \grad_{ F_s(y) } B_\theta,\R v\),$$ where
$\R v\subset T_{F_s(y)}\til{{N}}$ denotes the line through $v$.
Define $\tau_v$ to be the measure on $[0,\pi/2]$ given by
$$\tau_v(U)=\sigma_y^s\(\set{\theta\in\pa \til{{N}}\,|\,
t_v(\theta)\in U}\)$$
for any Borel subset $U\subset [0,\pi/2]$.
\begin{lemma}\label{lem:smallA}
If for some $0<\eps<1$, we have $\norm{A^{-1}}\geq
\frac{1}{\eps\norm{\sigma_y^s}}$, then for some $v\in S_{F_s(y)}\til{{N}}$ we
have
$$\int_0^{\pi/2}\cos^2(t)\, d\tau_v(t)\geq (1-\eps) \norm{\sigma_y^s}.$$
\end{lemma}
\begin{proof}
By our standing assumptions, the eigenvalues of the symmetric tensor
$D_xdB_\theta$ are all at least $1$ except for the eigenvalue in the
eigendirection $\grad_x B_\theta$ which is $0$. Hence, choosing a basis
$\set{e_i}$ for $T_{F_s(y)}\til{{N}}$, we may write,
$$A=\ \int_{\pa \til{{N}}}
  O_{\theta}^*
  {\begin{pmatrix}
  \scs 0 & \scs 0 & \scs \cdots & \scs 0\\
  \scs 0 & \scs \la_2 & \scs \cdots &\scs 0 \\
  \scs \vdots & \scs \ddots & \scs \ddots &\scs \vdots\\
  \scs 0 & \scs \cdots & \scs 0 & \scs \la_n
  \end{pmatrix}}
   O_\theta\ d\sigma_y^s(\theta)$$
for some mapping into the orthogonal group $\theta\mapsto O_\theta\in
O(T_{F_s(y)}\til{{N}})$ where $\la_i\geq 1$ for $i=2,\dots,n$ and
$O_\theta(\grad_{ F_s(y) } B_\theta)=e_1$. Suppose for some $v\in
S_{F_s(y)}\til{{N}}$, we have $\norm{A(v)}\leq a\eps \norm{\sigma_y^s}$. We
note that
$$\angle_{F_s(y)}\( e_1,O_\theta(\R v)\)=\angle_{F_s(y)}\( \grad_{ F_s(y) }
B_\theta,\R v\)=t(\theta).$$ Now we underestimate each $\la_i$ by replacing it
with $1$. In particular,
\begin{align*}
\inner{v,Av}&\geq \inner{v,\(\int_{\pa \til{{N}}} \Id
-\grad_{ F_s(y) } B_\theta^*\grad_{ F_s(y) } B_\theta\ d\sigma_y^s(\theta)\)(v)}\\
&=a \int_{\pa \til{{N}}}1- \inner{v,\grad_{ F_s(y) } B_\theta}^2\
d\sigma_y^s(\theta).
\end{align*}
Hence we have
$$\int_0^{\pi/2}1-\cos^2(t) d\tau(t)\leq \eps \norm{\sigma_y^s}.$$

\end{proof}

We now consider a $g$-orthonormal basis for $T_yM$ and a
$g_o$-orthonormal basis for $T_{F_s(y)}{N}$, so that we may discuss
the magnitude of $H$ with respect to these two metrics. First we
need another lemma.

\begin{lemma}\label{lem:sing}
  Suppose $(X,\mu)$ is a probability space and $u,v:X\to
  S^{n-1}\subset\R^n$ are measurable maps to the unit sphere. Then
  the singular values $0\leq\la_1\leq\dots\leq\la_n$ of the linear map $A:\R^n\to\R^n$ given
  by the (1,1)-tensor
  $$A=\int_X u(x)\tensor v(x)^*d\mu(x)$$
  satisfy
  $$\sum_{i=1}^n \la_i= \sup_{O\in O(n)}\int_X \inner{u(x),O(v(x))}d\mu(x)\leq
  1.$$
\end{lemma}
\begin{proof}
  Let $A=UDV$ be the singular value decomposition for $A$ where
  $D=\op{diag}(\la_1,\dots,\la_n)$ and $U,V$ are orthogonal. We have
  \begin{gather*}
  \tr D=\sup_{O\in O(n)}\tr \[DO\]=\sup_{O\in O(n)}\tr \[(U^*AV^*)O\]=\sup_{O\in O(n)}
  \tr \[AV^*OU\]=\sup_{O\in O(n)}\tr \[AO\]\\
  =\sup_{O\in O(n)}\int_X \tr \[u(x)\tensor
  O^*(v(x))\]d\mu(x)=\sup_{O\in O(n)}\int_X \inner{ u(x),O(v(x))}d\mu(x)\leq 1.
  \end{gather*}

\end{proof}

\begin{prop}\label{prop:singvalH}
If $\norm{A^{-1}}\geq \frac{1}{\eps\norm{\sigma_y^s}}$ for any $\eps\leq
1$, then the singular values of $H$ satisfy $\la_i\leq
\sqrt{\eps}\norm{\sigma_y^s}$ for $i=1,\dots,n-1$ and $\la_n\leq
\norm{\sigma_y^s}$.
\end{prop}
\begin{proof}
Let
$$0\leq \la_1\leq \dots<\la_n\leq \norm{\mu_y^s}$$ be the singular values of $H$ given by
$\la_i=\inner{w_i,H(u_i)}$ for a $g$-orthonormal frame $\set{u_i}$
and a $g_o$-orthonormal frame $\set{w_i}$. We can therefore have
$$\la_i=\int_{\widetilde{M}}\int_{\D \til{{N}}}
\inner{w_i,\grad_{ F_s(y) } B_\theta}\inner{\grad_y r_z,u_i}
\,d\nu_{\tilde{f}(z)}(\theta)\, d\mu_y^s(z).$$

Since $\inner{\grad_y r_z,u_i}\leq 1$ we can estimate
$$\la_i\leq\int_{\D \til{{N}}}
\inner{w_i,\grad_{ F_s(y) } B_\theta}\,d\sigma_y^s(\theta).$$

Now let $v$ be the vector provided by Lemma \ref{lem:smallA}, and write
$\grad B_\theta^\perp$ for the unit vector along the projection of $\grad_{
F_s(y) } B_\theta$ to $v^\perp$. then
\begin{align*}
\la_i&\leq\int_{\D \til{{N}}}
\inner{w_i,\grad_{ F_s(y) } B_\theta}\,d\sigma_y^s(\theta)\\
&=\int_{\D \til{{N}}}
\inner{w_i,\cos(t(\theta))v+\sin(t(\theta))\grad B_\theta^\perp}\,d\sigma_y^s(\theta)\\
&\leq\int_{0}^{\frac{\pi}{2}}
\inner{w_i,v}\cos(t)+\norm{\op{proj}_{v^\perp}(w_i)}\sin(t)\,d\tau(t).
\end{align*}

Since $\R w_n\directsum \R w_{n-1}$ intersects the subspace $v^\perp$ we
have
$$\la_{n-1}=\sup_{w\in w_n^\perp}
\inner{H^*(w),H^*(w)}^{\frac12}\leq \sup_{w\in v^\perp}
\inner{H^*(w),H^*(w)}^{\frac12}\leq\sup_{\substack{w\in v^\perp\\ u\in
S_y\til{M}}} \inner{w,H(u)}$$ where the last inequality holds since we may take
$u=\frac{H^*(w)}{\norm{H^*(w)}}.$ Hence we have from the previous
computation
$$\frac{\la_{n-1}}{\norm{\sigma_y^s}}\leq\frac{1}{\norm{\sigma_y^s}}\int_{0}^{\frac{\pi}{2}}
\sin(t)\,d\tau(t)\leq
\sqrt{\frac{1}{\norm{\sigma_y^s}}\int_{0}^{\frac{\pi}{2}}
\sin^2(t)\,d\tau(t)}\leq \sqrt{\eps}.$$ Here we have used H\"{o}lder's
inequality followed by Lemma \ref{lem:smallA}. Therefore, all $\la_i\leq
\sqrt{\eps}\norm{\sigma_y^s}$ for $i=1,\dots,n-1$, and so $\la_n\leq
\norm{\sigma_y^s}-\sum_{i=1}^{n-1}\la_i$ by applying Lemma \ref{lem:sing}
to the normalized measure $\frac{\sigma_y^s}{\norm{\sigma_y^s}}$.

\end{proof}

\begin{proof}[Proof of Theorem \ref{thm:gradest}]
Recall that we are assuming $\kappa(g)=1$ and we set
$\eps=\max\set{\frac{\norm{\mu_y^s}}{\norm{A^{-1}}},1}$. Observe that the
matrix associated to $(\det H) H^{-1}$ is the adjunct of the matrix associated to
$H$. By Proposition \ref{prop:singvalH}, its singular values are at most
$\sqrt{\eps}^{n-2} \norm{\mu_y^s}^{n-1}$ and one singular value is at most
$\sqrt{\eps}^{n-1}\norm{\mu_y^s}^{n-1}$.

First we show that the product $\circled{1}\cdot\circled{2}$ has norm
bounded above by $C(1+s)\(1+\eta^{-n}\)\sqrt{\eps}^{n-5}$ where $C$
depends only on $n$. Given the estimate of $\grad_v H$, it just remains to point
out that $\frac{\det H}{\det A}\norm{ H^{-1}}$ is bounded above by
$\frac{1}{\norm{\mu_y^s}}$ whenever $n\geq 4$.

On the other hand, $\circled{1}\cdot \circled{3}$ has norm bounded by
$\frac{\det H}{\det A} \norm{A^{-1}\grad_v A}$. This is bounded by
$\frac{\det H}{\det A} \(4s\beta \norm{A^{-1}}\norm{\mu_y^s}+s\)$. This is
in turn bounded by $4s\beta \sqrt{\eps}^{n-5}+s\sqrt{\eps}^{n-3}$ which is
bounded by $5s\beta$ whenever $n>4$.
\end{proof}

\section{Applications}\label{sec:applications}

In this section we will explore some of the consequences of theorem
\ref{thm:critical}. We first mention a couple of well-known
topological conditions for the existence of a map $f:M\to N$ of
nonzero degree. Let $\bar{N}$ be the cover of $N$ corresponding to
$f_*\pi_1(M)<\pi_1(N)$. Since $f$ induces a map
$\bar{f}:M\to\bar{N}$, we have the following commutative triangle,
$$\xymatrix{
  H_n(M) \ar[dr]_{\times \deg(f)} \ar[r]^{\times k}
                & H_n(\bar{N}) \ar[d]^{\times \[f_*\pi_1(M)\,:\,\pi_1(N)\]}  \\
                & H_n(N)             }$$
where the multiplication is with respect to the bases of $\Z$
determined by the respective fundamental classes and $k\in \Z$. In
particular, the index of $f_*\pi_1(M)$ in $\pi_1(N)$ divides
$\deg(f)$. Consequently $\pi_1(M)$ is virtually at least as large as
$\pi_1(N)$.

In the case $f$ has degree one, one may also deduce that
$f_*:H_*(M)\to H_*(N)$  is a split surjection using the induced map
$f^*:H^{n-*}(N)\to H^{n-*}(M)$ on cohomology together with
Poincar\'{e} duality. There is a similar statement for general
degree as well. Thus one can obtain obstructions from both the
homology groups, $H_*(M)$ and $H_*(N)$, and the fundamental groups,
$\pi_1(M)$ and $\pi_1(N)$, to the existence of a nonzero degree map.

For Theorems \ref{thm:ent-vol} and \ref{thm:critical} we would like
to obtain some bounds on $h(g)$ in terms of other quantities.

For $\Ga=\pi_1(M)$ and $S$ a finite subset of $\Ga$, let
$\inner{S}<\Ga$ be the subgroup generated by $S$. Let $\phi_S$ be
the metric on the Cayley graph of $(<S>,S)$ which is the weighted
simplicial distance where the length of each edge corresponding to a
generator $\sigma\in S$ is given by the Riemannian distance
$d(p,\sigma p)$ in the universal cover $(\tilde{M},g)$. Define
$$h_{g,S}=\limsup_{R\to\infty}\frac{\log
\#\set{\ga\in\Ga\,:\,\phi_S(\id,\ga)\leq R}}{R}.$$ Manning proved in
\cite{Manning:05a} the following formula for the volume growth
entropy,
$$h(g)=\sup\set{h_{g,S}\,:\,S \text{ finitely generates }\Ga}.$$
This allows us to obtain a curvature and entropy free restatement of Theorem
\ref{thm:critical}. Here the entropy is replaced by a dilatation.

\begin{cor}
For any $(N,g_o)\in \mc{N}_{n,\rho}$ with $u(g_o)=h(g_o)$ and $n>4$ and
closed Riemannian manifold $(M,g)$ together with any continuous map $f:M\to
N$, there exists a $C^1$ map $F:M\to N$ homotopic to $f$ and an $r>0$ such
that
$$\vol_{g}(M)\geq \op{H}^n \abs{\deg(f)}\vol_{g_o}(N) + \frac12 \vol_g\(T_{\frac{r}{\kappa(g)}}(F)\),$$
where $$\op{H}=\inf\set{\frac{h_{g_o,f_*S}}{h_{g,S}}\,:\,S \text{ finitely
generates }\pi_1(M)}.$$ As before, $r$ depends only on
$n,\rho,\abs{\deg(f)}\norm{N}$ and $\kappa(g)\injrad{\til{g}}$.
\end{cor}

\begin{proof}
Take a sequence $S_i$ so that  $h_{g,S_i}\to h(g)$ from below. Note
that $u(g_o)=h(g_o)\geq h_{g_o,S'_i}\geq h(g_o,f_*S_i)$ where $S'_i$
is any generating set for $\pi_1(N)$ containing
$f_*S_i=\set{f_*\sigma\,:\,\sigma\in S_i}$. Hence after setting
$\eps_i=h(g)-h_{g,S_i}$, we have
$$\(h_{g,S_i}+\eps_i\)^n\(\vol_{g}(M)-\frac12 \vol_g\(T_{\frac{r}{\kappa(g)}}(F)\)\)
\geq h_{g_o,f_*S_i}^n\abs{\deg(f)}\vol_{g_o}(N).$$ Letting
$\eps_i\to 0$ and noting that
$H\leq\liminf_i\frac{h_{g_o,f_*S_i}}{h_{g,S_i}}$ finishes the proof.

\end{proof}

Consider a smooth $n$-manifold $(N,g_o)$ with a codimension $0$ smooth
submanifold $S$, and suppose $M_1=N\setminus S$ and $M_2$ is another
smooth $n$-manifold with boundary admitting a map $f:M_2\to S$ which is a
diffeomorphism from $\partial M_2$ to $\partial S$. Let $\pa f$ denote the
restriction of $f$ to $\pa M_2$. There is a degree one map from the adjunction
space $M_1\cup_{\pa f} M_2$ to $N$ formed by crushing $M_2$ to $S$ via $f$.

\begin{corollary}
  Fix $(N,g_o)$ with $-\rho^2\leq K_{g_o}\leq -1$. For any smooth metric $g$
  on $M_1\cup_{\pa f} M_2$ rescaled so that $K_g\geq -1$, there is a constant $C(n,\injrad(g),\rho)$ such
  that if
  $$\vol_{g}(M_1\cup_{\pa f} M_2)\leq \vol_{g_o}(N)+C,$$ then $M$ is
  diffeomorphic to $N$.
\end{corollary}

A special case of this is Corollary \ref{cor:connect} where such an
$M$ is $M=N\# Q$ for any smooth $n$-manifold $Q$. This is
interesting in the constext of the following result of Farrell and
Jones demonstrating that smooth rigidity fails in the negatively
curved category.

\begin{theorem}[\cite{Farrell-Jones:89c}]\label{thm:smooth-exotic} Let $n> 5$ be any
dimension for which there exist distinct projectively inequivalent
exotic smooth spheres $\Sigma_1,\cdots,\Sigma_k$. for any hyperbolic
$n$-manifold $M$ and any $\delta>0$ there is a finite cover
$\hat{M}$ of $M$ such that
$\hat{M},\hat{M}\#\Sigma_1,\hat{M}\#\Sigma_2,\dots,\hat{M}\#\Sigma_k$
are pairwise homeomorphic but not diffeomorphic. Moreover for
$i=1,\dots,k$, there are metrics $g_i$ on $M\#\Sigma_i$ such that
$$-1-\delta\leq  K_{g_i} \leq -1.$$
\end{theorem}

By another construction, the conclusion of the above theorem also
holds in dimensions $n>4$. This theorem has been extended by Farrell
separately with Jones (\cite{Farrell-Jones:94}) and Aravinda
(\cite{Aravinda-Farrell:04}) to show that there are closed manifolds
of almost quarter pinched negative curvature which are homeomorphic
but not diffeomorphic to complex hyperbolic and quaternionic
hyperbolic manifolds. Observe that Theorem \ref{thm:vol} implies
that the degree of the cover in this theorem must depend on the
choice of $\delta$.

Farrell and Jones (\cite{Farrell-Jones:94c}) also gave a set of four
criteria, in terms of a boundary conjugacy, for when an isomorphism
between fundamental groups of two nonpositively curved manifold may
be realized by a diffeomorphism.
\begin{theorem}[\cite{Farrell-Jones:94c}]
Given nonpositively curved closed Riemannian manifolds $(M,g)$ and
$(N,g_o)$ and an isomorphism $\alpha:\pi_1(M)\to\pi_1(N)$, there
exists a diffeomorphism $f:M\to N$ with $f_*=\alpha$ provided:
\begin{enumerate}

  \item $\partial_\infty\tilde{M}$ and $\partial_\infty\tilde{N}$ have a natural $C^1$
  structure
  \item There is a $C^1$ conjugacy $\bar{h}:\partial_\infty\tilde{M}
  \goto{\cong}\partial_\infty\tilde{N}$ of the $C^1$ actions of $\pi_1(M)$ on
$\partial_\infty\tilde{M}$ and $\pi_1(N)$ on
$\partial_\infty\tilde{N}$. I.e.
$\bar{h}\of\gamma=\alpha(\gamma)\bar{h}$ for all
$\gamma\in\pi_1(M).$

\item The conjugacy $\bar{h}$ extends to a $C^0$ semiconjugacy $\tilde{h}:\tilde{M}\to
\tilde{N}$, I.e. the lift of a continuous map $h:M\to N$.

\item $\chi(M)=0$ (e.g. $n$ is odd).
\end{enumerate}
\end{theorem}

The principal drawback to applying this theorem is that $C^1$
structures on $\pa \til{M},\pa \til{N}$ are only known to exist when
$M$ and $N$ either are of higher rank, are locally rank one
symmetric spaces, or are quarter pinched and negatively curved. In
the negatively curved case a $C^\infty$ structure would imply that
$M$ and $N$ are locally symmetric by
\cite{Benoist-Foulon-Labourie:92}. In light of this, Theorems
\ref{thm:vol} and \ref{thm:ent-vol} can be viewed as a an
effectively computable gap criterion for smooth equivalence.

Now we turn to proving some of the corollaries mentioned in the
introduction.

\begin{proof}[Proof of Corollary \ref{cor:pinch}]

If the conclusion does not hold, then there is a sequence of
pairwise homotopy equivalent nondiffeomorphic manifolds $(M_i,g_i)$
with $-1-\delta_i\leq  K_{g_i} \leq -1$ for a sequence $\delta_i\to
0$. The isomorphisms $(f_{ij})_*:\pi_1(M_i)\to\pi_1(M_j)$ are
induced by continuous maps $f_{ij}$ which can be chosen to be of
degree $1$ since the $M_i$ are $K(\pi_1(M_i),1)$'s. %Moreover after a

We can rescale the metric $g_i^\prime=(1+\delta_i)g_i$ so that
$-1\leq  K_{g_i^\prime} \leq \frac{-1}{1+\delta_i}$. The real
Schwarz lemma of \cite{Besson-Courtois-Gallot:99} then gives
$$(1+\delta_i)^{\frac{n}2}\vol_{g_i}(M_i)=\vol_{g_i^\prime}({M_i})
\geq
\vol_{g_j}(M_j)=\frac{\vol_{g_j^\prime}(M_j)}{(1+\delta_j)^n}.$$

This holds for all $i,j$, so the volumes form a Cauchy sequence.
Moreover, this provides a bound on the injectivity radius as
follows. If the injectivity radius of $(M_i,g_i^\prime)$ is less
than the Margulis constant $\eps$, then any component $A^o_{i,\eps}$
of the ``thin'' set $A_{i,\eps}:=\set{x\in M_i\,|\, \injrad(x)<
\eps}$ consists of a uniform tube neighborhood of a geodesic $\ga$
with $g_i^\prime$-length $l_i<\eps$. The volume of this tube
satisfies $\vol_{g_i^\prime}(A^o_{i,\eps})\geq C \abs{\log l_i}$
(see the discussion section in \cite{Reznikov:95}). However
$\vol_{g_i^\prime}(M_i)$ is bounded above, and therefore the
injectivity radius is bounded from below independent of $i$. In
particular, the $C$ of theorem \ref{thm:vol} depends only on $n$ for
this sequence of $M_i$ which contradicts their volumes converging.

\end{proof}

From this we can obtain Theorem \ref{thm:Belegradek} in the compact
case and $n>4$.
\begin{proof}[Proof of Theorem \ref{thm:Belegradek}]
We can imitate the proof above, to obtain that the family of pinched
negatively curved closed manifolds with fixed $\pi_1(M)$ has
uniformly bounded volume from above, and injectivity radius from
below. We then can choose a set of manifolds whose volumes are
within the constant $C$ of any manifold in this class. This covering
number bounds the number of diffeomorphism classes.
\end{proof}

\begin{proof}[Proof of Corollary \ref{cor:finiteness}]
Let $\mc{N}$ be the class of all the $(N,g_o)$, that is all smooth structures and
metrics on the fixed topological manifold $N$, satisfying the hypotheses. By
definition, the members of $\mc{N}$ are pairwise homeomorphic. In particular,
each element of $\mc{M}$ admits a degree one map to all of the elements of
$\mc{N}$. In particular there is an $(N,g_o)\in\mc{N}$ such that
$\vol_{g}(M)\leq \vol_{g_o}(N)+C$. Thus we can apply Theorem \ref{thm:vol}
to conclude that $M$ and $N$ are diffeomorphic. Since $\mc{N}$ is a class with
only a finite number of diffeomorphism types, so is $\mc{M}$.
\end{proof}
\begin{remark}
  The above easily generalized to the case where $\mc{N}$ is any class of
  negatively curved manifolds with a finite number of diffeomorphism  types
  and each member of $\mc{M}$ maps onto at least element of $\mc{N}$, but
  then the formula for the entropy-volume bound $V(\delta)$ for each $(M,g)\in \mc{M}$ must be
  restricted over those elements of $\mc{N}$ admitting a degree one map from
  $M$.
\end{remark}

\begin{remark}
 An important point to the above proofs is that we do not need to use
  Cheeger finiteness (or any other form of Gromov-Hausdorff
  compactness theorems). Instead we have replaced this step with the
  direct analytic argument of Theorem \ref{thm:vol}).
\end{remark}

\begin{example}\label{ex:conditions}
We first point out that the constant $C$ in Theorems \ref{thm:vol}
and \ref{thm:ent-vol} as well as $r$ in Theorem \ref{thm:critical}
must depend on $\delta$. Otherwise, we could take a hyperbolic
$(N,g_o)$ and let $M=N\#Q$ the smooth connect sum with a very small
diameter Riemannian manifold $(Q,g_Q)$ with arbitrary topology.
However, $M$ admits a degree $1$ map to $N$ and by scaling $g_Q$,
$\vol_{g}(M)$ can be made as close to $\vol_{g_o}(N)$ as desired.

Similarly, the dependence of $C$ and $r$ on $\rho$ is also necessary. Take a
fixed hyperbolic manifold $(M,g)$ of finite volume and a sufficiently small
injectivity radius $\delta$. By the Margulis Lemma, there is an $\eps_o>0$ such
that each component $A_\eps$ of the $\eps$-thin part $M_\eps$ consisting of
points with injectivity radius less than $\eps$ for $\delta<\eps<\eps_o$ is
topologically an $n-1$ ball bundle over $S^1$, where the $0$ section is a short
geodesic. The metric is locally $\cosh(r)^2 d\ga^2+dr^2+\sinh(r)^2
d\theta^2$ where $d\gamma$ is the geodesic arc element, $dr$ the radial
element and $d\theta^2$ represents the combined spherical metric. For
$\eps^{\prime}<\eps$ let $g^{\prime}$ be a new metric on $M$ such that
$g^{\prime}=g$ on the complement of $A_\eps$ and on $A_{\eps^{\prime}}$,
$g^{\prime}$ has constant curvature $-\rho^2$. We can achieve this with a
warped product metric of the form $h(r)^2 d\ga^2+dr^2+(h(r)^2-h(0)^2)
d\theta^2$ for a convex function which interpolates between $\cosh(r)$ for
$r>\op{diam}(A_\eps)$ and $\frac{\cosh(\rho r)}{\rho}$ for $r\leq
\op{diam}(A_\eps^{\prime})$ with uniformly bounded
$0\frac{h^{\prime}}{h}\leq 0$ and $\frac{h^{\prime\prime}}{h}$. For any
$\la>0$ we can choose $\eps$ sufficiently close to $\delta$ so that
$\vol_{g}(A_\eps)<\la$.  Now $(M,\rho^2 g^{\prime})$ has curvature bounded
below by $-1$ and constant curvature $-1$ on the tube $A_{\eps^{\prime}}$. If
the injectivity radius of this rescaled metric is not sufficiently large on
$A_{\eps^{\prime}}$ we can replace $M$ by a cover so that the injectivity
radius of $(A_{\eps^{\prime}},\rho^2g^{\prime})$ is sufficiently large to
perform the operation described in \cite{Farrell-Jones:89c}. Note that since
$(A_{\eps^{\prime}},\rho^2g^{\prime})$ has constant curvature $-1$, the
degree of the cover is independent of $\rho$. Also assuming that we have
chosen an appropriate dimension $n$, there is an exotic smooth $n$-sphere
$\sigma$ such that by Theorem \ref{thm:smooth-exotic}, $N=M\#\Sigma$ is
not diffeomorphic to $M$. Observe also that the connect sum and the
procedure of controlling the curvature, Proposition 1.3 of
\cite{Farrell-Jones:89c}, are carried out precisely in a tube of bounded
diameter with the metric isometric to the original in a neighborhood of the
boundary. Since the $\rho^2 g^{\prime}$ diameter of $A_{\eps^{\prime}}$ is
unbounded as $\rho\to \infty$, we conclude that there exists a metric, denoted
by $\rho^2 g_o$, on $N$ which agrees with $\rho^2g^{\prime}$ outside of
$A_{\eps^{\prime}}\#\Sigma$ and $-\frac32\leq
 K_{\rho^2g_o} \leq -\frac12$ on $A_{\eps^{\prime}}$. In particular
$-\frac32 \rho^2 K_{g_o} \leq -1$. Summarizing we have $\vol_{g}(M)\leq
Vol(N,g_o)+\la$, and $N$ not diffeomorphic to $N$. Now we can choose $\la$
to be arbitrarily small which would violate Theorem \ref{thm:vol} if $C$ did not
depend on $\rho$.
\end{example}

\def\cprime{$'$}
\providecommand{\bysame}{\leavevmode\hbox to3em{\hrulefill}\thinspace}
\providecommand{\MR}{\relax\ifhmode\unskip\space\fi MR }
\providecommand{\MRhref}[2]{%
  \href{http://www.ams.org/mathscinet-getitem?mr=#1}{#2}
} \providecommand{\href}[2]{#2}


\begin{thebibliography}{BMOR84}

\bibitem[AC91]{Anderson-Cheeger:91a}
M.~T. Anderson and J.~Cheeger, \emph{Diffeomorphism finiteness for
manifolds
  with {R}icci curvature and {$L\sp {n/2}$}-norm of curvature bounded}, Geom.
  Funct. Anal. \textbf{1} (1991), no.~3, 231--252.

\bibitem[AC92]{Anderson-Cheeger:92a}
\bysame, \emph{{$C\sp \alpha$}-compactness for manifolds with {R}icci
curvature
  and injectivity radius bounded below}, J. Differential Geom. \textbf{35}
  (1992), no.~2, 265--281.

\bibitem[AF04]{Aravinda-Farrell:04}
C.~S. Aravinda and F.~T. Farrell, \emph{Exotic structures and quaternionic
  hyperbolic manifolds}, Algebraic groups and arithmetic, Tata Inst. Fund.
  Res., Mumbai, 2004, pp.~507--524.

\bibitem[Ano69]{Anosov:69}
D.~V. Anosov, \emph{Geodesic flows on closed {R}iemann manifolds with
negative
  curvature.}, American Mathematical Society, Providence, R.I., 1969.

\bibitem[Bal95]{Ballmann:95}
W.~Ballmann, \emph{Lectures on spaces of nonpositive curvature}, DMV
Seminar,
  vol.~25, Birkh\"auser Verlag, Basel, 1995, With an appendix by Misha Brin.

\bibitem[BCG95]{Besson-Courtois-Gallot:95}
G.~Besson, G.~Courtois, and S.~Gallot, \emph{Entropies et rigidit\'es des
  espaces localement sym\'etriques de courbure strictement n\'egative}, Geom.
  Funct. Anal. \textbf{5} (1995), no.~5, 731--799.

\bibitem[BCG96]{Besson-Courtois-Gallot:96}
G.~Besson, G.~Courtois, and S.~Gallot, \emph{Minimal entropy and {M}ostow's
  rigidity theorems}, Ergodic Theory Dynam. Systems \textbf{16} (1996), no.~4,
  623--649.

\bibitem[BCG98]{Besson-Courtois-Gallot:98}
G.~Besson, G.~Courtois, and S.~Gallot, \emph{A real {S}chwarz lemma and
some
  applications}, Rend. Mat. Appl. (7) \textbf{18} (1998), no.~2, 381--410.

\bibitem[BCG99]{Besson-Courtois-Gallot:99}
G.~Besson, G.~Courtois, and S.~Gallot, \emph{Lemme de {S}chwarz r\'eel et
  applications g\'eom\'etriques}, Acta Math. \textbf{183} (1999), no.~2,
  145--169.

\bibitem[Bel02]{Belegradek:02}
I.~Belegradek, \emph{On {M}ostow rigidity for variable negative curvature},
  Topology \textbf{41} (2002), no.~2, 341--361.

\bibitem[Bes98]{Bessieres:98}
L.~Bessi{\`e}res, \emph{Un th\'eor\`eme de rigidit\'e diff\'erentielle},
  Comment. Math. Helv. \textbf{73} (1998), no.~3, 443--479.

\bibitem[Bes00]{Bessieres:00a}
L.~Bessi{\`e}res, \emph{Sur le volume minimal des vari\'et\'es ouvertes}, Ann.
  Inst. Fourier (Grenoble) \textbf{50} (2000), no.~3, 965--980.

\bibitem[BFL92]{Benoist-Foulon-Labourie:92}
Y.~Benoist, F.~Foulon, and F.~Labourie, \emph{Flots d'{A}nosov {\`a}
  distributions stable et instable diff{\'e}rentiables}, 33--74.

\bibitem[BK78]{Bourguignon-Karcher:78}
J.-P. Bourguignon and H.~Karcher, \emph{Curvature operators: pinching
estimates
  and geometric examples}, Ann. Sci. \'Ecole Norm. Sup. (4) \textbf{11} (1978),
  no.~1, 71--92.

\bibitem[BMOR84]{Bemelmans-Min-Oo-Ruh:84}
J.~Bemelmans, Min-Oo, and E.~A. Ruh, \emph{Smoothing {R}iemannian
metrics},
  Math. Z. \textbf{188} (1984), no.~1, 69--74.

\bibitem[Cha93]{Chavel93}
I.~Chavel, \emph{Riemannian geometry---a modern introduction}, Cambridge
Tracts
  in Mathematics, vol. 108, Cambridge University Press, Cambridge, 1993.

\bibitem[Che69]{Cheeger:69}
J.~Cheeger, \emph{Pinching theorems for a certain class of {R}iemannian
  manifolds}, Amer. J. Math. \textbf{91} (1969), 807--834.

\bibitem[Che70]{Cheeger:70a}
\bysame, \emph{Finiteness theorems for {R}iemannian manifolds}, Amer. J.
Math.
  \textbf{92} (1970), 61--74.

\bibitem[Cro80]{Croke:80}
C.~B. Croke, \emph{Some isoperimetric inequalities and eigenvalue estimates},
  Ann. Sci. \'Ecole Norm. Sup. (4) \textbf{13} (1980), no.~4, 419--435.

\bibitem[Cro88]{Croke:88}
\bysame, \emph{An isoembolic pinching theorem}, Invent. Math. \textbf{92}
  (1988), no.~2, 385--387.

\bibitem[EH90]{Eschenburg-Heintze90}
J.-H. Eschenburg and E.~Heintze, \emph{Comparison theory for {R}iccati
  equations}, Manuscripta Math. \textbf{68} (1990), no.~2, 209--214.

\bibitem[FJ89a]{Farrell-Jones:89c}
F.~T. Farrell and L.~E. Jones, \emph{Negatively curved manifolds with exotic
  smooth structures}, J. Amer. Math. Soc. \textbf{2} (1989), no.~4, 899--908.

\bibitem[FJ89b]{Farrell-Jones:89a}
\bysame, \emph{A topological analogue of {M}ostow's rigidity theorem}, J.
Amer.
  Math. Soc. \textbf{2} (1989), no.~2, 257--370.

\bibitem[FJ90]{Farrell-Jones:90}
F.~T. Farrell and L.~E. Jones, \emph{Classical aspherical manifolds}, CBMS
  Regional Conference Series in Mathematics, vol.~75, Published for the
  Conference Board of the Mathematical Sciences, Washington, DC, 1990.

\bibitem[FJ93]{Farrell-Jones:93}
F.~T. Farrell and L.~E. Jones, \emph{Topological rigidity for compact
  non-positively curved manifolds}, Differential geometry: Riemannian geometry
  (Los Angeles, CA, 1990), Proc. Sympos. Pure Math., vol.~54, Amer. Math. Soc.,
  Providence, RI, 1993, pp.~229--274.

\bibitem[FJ94a]{Farrell-Jones:94}
\bysame, \emph{Complex hyperbolic manifolds and exotic smooth structures},
  Invent. Math. \textbf{117} (1994), no.~1, 57--74.

\bibitem[FJ94b]{Farrell-Jones:94c}
\bysame, \emph{Smooth rigidity and {$C\sp 1$}-conjugacy at {$\infty$}},
Comm.
  Anal. Geom. \textbf{2} (1994), no.~4, 563--578.

\bibitem[FJO98]{Farrell-Jones-Ontaneda:98b}
F.~T. Farrell, L.~E. Jones, and P.~Ontaneda, \emph{Hyperbolic manifolds with
  negatively curved exotic triangulations in dimensions greater than five}, J.
  Differential Geom. \textbf{48} (1998), no.~2, 319--322.

\bibitem[Fuk84]{Fukaya:84}
K.~Fukaya, \emph{A finiteness theorem for negatively curved manifolds}, J.
  Differential Geom. \textbf{20} (1984), no.~2, 497--521.

\bibitem[Gro78]{Gromov:78b}
M.~Gromov, \emph{Manifolds of negative curvature}, J. Differential Geom.
  \textbf{13} (1978), no.~2, 223--230.

\bibitem[Gro82a]{Gromov82}
\bysame, \emph{Volume and bounded cohomology}, Inst. Hautes \'Etudes Sci.
Publ.
  Math. (1982), no.~56, 5--99 (1983).

\bibitem[Gro82b]{Gromov:82}
M.~Gromov, \emph{Volume and bounded cohomology}, Inst. Hautes \'Etudes
Sci.
  Publ. Math. (1982), no.~56, 5--99 (1983).

\bibitem[GT87]{Gromov-Thurston:87}
M.~Gromov and W.~Thurston, \emph{Pinching constants for hyperbolic
manifolds},
  Invent. Math. \textbf{89} (1987), no.~1, 1--12.

\bibitem[HIH77]{Heintze-Hof77}
E.~Heintze and H.~Im~Hof, \emph{Geometry of horospheres}, J. Differential
Geom.
  \textbf{12} (1977), no.~4, 481--491 (1978).

\bibitem[Hir76]{Hirsch:76}
M.~W. Hirsch, \emph{Differential topology}, Springer-Verlag, New York, 1976,
  Graduate Texts in Mathematics, No. 33.

\bibitem[Kap05]{Kapovitch:05}
V.~Kapovitch, \emph{Curvature bounds via {R}icci smoothing}, Illinois J. Math.
  \textbf{49} (2005), no.~1, 259--263 (electronic).

\bibitem[Man79]{Manning:79}
A.~Manning, \emph{Topological entropy for geodesic flows}, Ann. of Math. (2)
  \textbf{110} (1979), no.~3, 567--573.

\bibitem[Man05]{Manning:05a}
A.~Manning, \emph{Relating exponential growth in a manifold and its
fundamental
  group}, Proc. Amer. Math. Soc. \textbf{133} (2005), no.~4, 995--997
  (electronic).

\bibitem[Pet98]{Petersen98}
P.~Petersen, \emph{Riemannian geometry}, Graduate Texts in Mathematics,
vol.
  171, Springer-Verlag, New York, 1998.

\bibitem[Rez95]{Reznikov:95}
A.~Reznikov, \emph{The volume and the injectivity radius of a hyperbolic
  manifold}, Topology \textbf{34} (1995), no.~2, 477--479.

\bibitem[Ron96]{Rong:96a}
X.~Rong, \emph{On the fundamental groups of manifolds of positive sectional
  curvature}, Ann. of Math. (2) \textbf{143} (1996), no.~2, 397--411.

\bibitem[Shi89]{Shi:89}
W.-X. Shi, \emph{Deforming the metric on complete {R}iemannian manifolds}, J.
  Differential Geom. \textbf{30} (1989), no.~1, 223--301.

\bibitem[Thu77]{Thurston:77}
W.~P. Thurston, \emph{Geometry and $3$-manifolds}, (a.k.a. Thurston's
Notes),
  1977.

\bibitem[Wan72]{Wang:72}
H.~C. Wang, \emph{Topics on totally discontinuous groups}, Symmetric spaces
  (Short Courses, Washington Univ., St. Louis, Mo., 1969--1970), Dekker, New
  York, 1972, pp.~459--487. Pure and Appl. Math., Vol. 8.

\end{thebibliography}
\end{document}